\documentclass[11pt]{amsart}
\usepackage{amssymb,amscd}
\usepackage{color}
\usepackage{amsmath}
\usepackage{enumerate}
\usepackage{amsfonts}
\usepackage{graphicx}
\usepackage{amsmath,amsthm,amssymb}
\usepackage{latexsym}
\usepackage{subfigure}
\usepackage{psfrag}
\usepackage{latexsym}
\usepackage{float}
\usepackage[latin1]{inputenc}
\theoremstyle{plain}
\newtheorem{theorem}{Theorem}
\newtheorem{lemma}[theorem]{Lemma}
\newtheorem{proposition}[theorem]{Proposition}

\newtheorem{definition}[theorem]{Definition}
\numberwithin{theorem}{section}
\numberwithin{equation}{section}
\numberwithin{table}{section}
\numberwithin{figure}{section}

\theoremstyle{remark}
\newtheorem*{remark}{Remark}

\begin{document}

\title[The Euclidean Picard modular groups]{Generators for the Euclidean Picard Modular Groups}
\author{Tiehong zhao}
\address{Institut de Math\'{e}matiques\\Universit\'{e} Pierre et Marie Curie\\ 4, place Jussieu\\F-75252 Paris, France}
\email{zhao@math.jussieu.fr}

\date{}
\thanks{The author is supported by the China-funded Postgraduates Studying Aboard Program for Building Top
University.}

\begin{abstract}
The goal of this article is to show that five explicitly given transformations, a rotation, two screw Heisenberg rotations,
 a vertical translation and an involution generate the Euclidean Picard modular groups with coefficient in the Euclidean ring
  of integers of a quadratic imaginary number field. We also obtain the relations of the isotropy subgroup by analysis of the
  combinatorics of the fundamental domain in Heisenberg group.
\end{abstract}

\maketitle

\section{Introduction}
\label{sec-intro}

Let $\mathcal{K}=\mathbb{Q}(\sqrt{-d})$ be a quadratic imaginary number field. Let $\mathcal{O}_d$ be the ring
of algebraic integers of $\mathcal{K}$. The Bianchi groups $PSL_2(\mathcal{O}_d)$ are the simplest numerically defined discrete groups. In number theory they have been used to study the zeta-functions of binary Hermitian forms over the rings $\mathcal{O}_d$. As the isometric groups acting on the half-upper space, they are of interest in the theory of Fuchsian groups and the related theory of Riemann surfaces. The Bianchi groups can be considered as the natural algebraic generalization of the classical modular group $PSL_2(\mathbb{Z})$. A good general reference for the Bianchi groups and their relation to the modular group is \cite{Fi}. As a natural generalization of the Bianchi groups, the subgroups of $PU(2,1)$ with coefficients in $\mathcal{O}_d$ are called {\it Picard modular groups}, denoted by $PU(2,1;\mathcal{O}_d)$. These groups have attracted a great deal of attention both
for their intrinsic interest as discrete groups and also for their applications in complex hyperbolic geometry
 (as the holomorphic automorphism subgroups).

A general method to determine finite presentations for each Bianchi group $PSL_2(\mathcal{O}_d)$ was developed by Swan \cite{Sw} based on geometrical work of Bianchi, while a separate purely algebraic method was given by Cohn \cite{Co}. In general, fundamental domains for Lie groups were studied by \cite{GR}, but the complex hyperbolic space is a particularly challenging case since no existence of totally geodesic hypersurface. So far very few examples of
 complex hyperbolic lattices have been constructed explicitly. Due to the famous paper \cite{Mo} of Mostow, other explicit
  constructions of fundamental domains for lattices in $PU(2,1$) were obtained, see for example, the work of Goldman
  and Parker \cite{GP}, Deraux, Falbel and Paupert \cite{DFP}, Schwartz \cite{Sch}, the survey paper of Parker \cite{Pa}. In particular,
  the group $PU(2, 1;\mathcal{O}_3) = PU(2, 1;\mathbb{Z}[\omega]),$
   where $\omega$ is a cube root of unity was studied by Falbel and Parker in \cite{FP} and its sister was treated recently in \cite{Zh}.
   Analogously a fundamental domain of Gauss-Picard group $PU(2,1;\mathbb{Z}[i])$ was described in (\cite{FFP}, \cite{FL1}, \cite{FL2}) and
   analysis of the combinatorics of the fundamental domain gives rise to a presentation of the group in \cite{FFP}.

In this paper we give a description of generators for certain Picard modular groups $PU(2,1;\mathcal{O}_d)$ where the
ring $\mathcal{O}_d$ is Euclidean except for $d=1,3$ (these two exceptional cases have been studied in many aspects).
Among the quadratic imaginary number rings $\mathcal{O}_d$ only $\mathcal{O}_1,\mathcal{O}_2,\mathcal{O}_3,\mathcal{O}_7,\mathcal{O}_{11}$
have a Euclidean algorithm, see \cite{ST}, although there is a larger finite collection of $\mathcal{O}_d$'s
(such as $d=1,$ $2,$ $3,$ $7,$ $11,$ $19,$ $43$, $67,$ $163$, see \cite{Zi}) which have class number one.
For these values of $d$ the orbifold $\mathbf{H}^2_\mathbb{C}/PU(2, 1;\mathcal{O}_d)$ has only one cusp.
 The method is based on the construction of various shapes of precisely fundamental domains for the stabilisers of infinity
 of $PU(2,1;\mathcal{O}_d)$ and then on a determination of several neighboring isometric spheres such that the union of the boundaries of these isometric
 spheres contains the fundamental domain of the stabiliser, which was used in (\cite{FP}, \cite{FFP}, \cite{Zh}). Compared with other groups, the generators of these groups in (\cite{FP}, \cite{FFP}, \cite{Zh}) are easy to be obtained since the fundamental domain constructed lies completely inside the
  boundary of the isometric sphere centred at origin. Again this reflects the underlying number theory; $\mathcal{O}_1$
  and $\mathcal{O}_3$ have non-trivial units while the other three do not. A simple algorithm to decompose any
  transformation in the Picard group $PU(2,1;\mathcal{O}_1)$ as a product of the generators was given in \cite{FFLP},
   one would be interesting to extend their method to other Picard modular groups. However, it would also be important to find
    the generators in terms of geometric ways which will provide more informations that one continue to construct a fundamental
  domain explicitly for each of Picard modular groups.

  I would like to thank my advisor E. Falbel for his warm encouragements all along
this work and for a number of helpful comments.

\section{Complex hyperbolic space}
\label{sec-complex}
\subsection{The Siegel domain}

\medskip

In this section we give the necessary background material on complex hyperbolic space. To know more details of this material we refer the reader to \cite{Go}.

Let $\mathbb{C}^{2,1}$ denote the complex vector space equipped with the Hermitian form defined by
\begin{equation*}\langle \mathbf{z},\mathbf{w}\rangle=z_1\bar{w}_3+z_2\bar{w}_2+z_3\bar{w}_1,
\end{equation*}
where $\mathbf{z}$ and $\mathbf{w}$ be the column vectors $[z_1,z_2,z_3]^t$ and $[w_1,w_2,w_3]^t$ respectively.
The projective model of complex hyperbolic space $\mathbf{H}_{\mathbb{C}}^2$ is defined to be the collection of negative
lines in $\mathbb{C}^{2,1}$, namely, those points $\mathbf{z}$ satisfying $\langle \mathbf{z},\mathbf{z}\rangle<0$.

We mainly take the {\it Siegel domain} $\mathfrak{S}$ as a upper half-space model for the complex hyperbolic space, that is given by
$$\mathfrak{S}=\{(z_1,z_2)\in\mathbb{C}^2:2\Re ez_1+|z_2|^2<0\}.$$ The boundary of the {\it Siegel domain} $\mathfrak{S}$ is identified with the one-point compactification of the Heisenberg group. The Heisenberg group $\mathfrak{R}$ is $\mathbb{C}\times \mathbb{R}$ with the group law
\begin{equation*}(\zeta_1,t_1)\diamond(\zeta_2,t_2)=(\zeta_1+\zeta_2,t_1+t_2+2\Im m(\zeta_1\bar{\zeta}_2)).\end{equation*}
The Cygan metric on $\mathfrak{R}$ is given by
$$
\rho_0((\zeta_1,t_1),(\zeta_2,t_2))=\left||\zeta_1-\zeta_2|^2-it_1+it_2-2i\Im m(\zeta_1\bar{\zeta}_2)\right|,
$$
in terms of the operation of Heisenberg group, that is $\left|(\zeta_1,t_1)^{-1}\diamond(\zeta_2,t_2)\right|$.

We can extend the Cygan metric to an incomplete metric on
$\bar{\mathfrak{S}}-\{\infty\}$ as follows
$$\tilde{\rho}_0=\left||\zeta_1-\zeta_2|^2+|u_1-u_2|-it_1+it_2-2i\Im m(\zeta_1\bar{\zeta}_2)\right|.$$

The Siegel domain $\mathfrak{S}$ is parametrised in horospherical
coordinates by
\begin{equation}\label{eq:2-1}(\zeta,t,u)\longrightarrow\left[\begin{array}{c}(-|\zeta|^2-u+it)/2\\ \zeta\\1\end{array}\right]\end{equation}
and the point at infinity being
\begin{equation*}q_{\infty}=\left[\begin{array}{c}1\\0\\0\end{array}\right].\end{equation*}
Then $\mathfrak{S}=\mathfrak{R}\times\mathbb{R}_+$ and $\partial\mathfrak{S}=(\mathfrak{R}\times\{0\})\cup\{q_{\infty}\}$.

\medskip

\subsection{Complex hyperbolic isometries}

\medskip

Let $U(2,1)$ be the group of matrices that are unitary with respect to the form $\langle .,.\rangle$. The group of holomorphic isometries of complex hyperbolic space is the projective unitary group
$PU(2,1)=U(2,1)/U(1)$, with a natural identification $U(1)=\{e^{i\theta}I,\theta\in[0,2\pi)\}.$ We now describe the action of the stabiliser of $q_\infty$ on the Heisenberg group.

The Heisenberg group acts on itself by {\it Heisenberg translations}. For $(\tau,v)\in\mathfrak{R}$, this is
\begin{equation*}T_{(\tau,v)}:(\zeta,t)\mapsto(\zeta+\tau,t+v+2\Im m(\tau\bar{\zeta}))=(\tau,v)\diamond(\zeta,t).\end{equation*}
Heisenberg translation by $(0,v)$ for any $v\in\mathbb{R}$ is called {\it vertical translation} by $v$.

The unitary group $U(1)$ acts on the Heisenberg group by {\it Heisenberg rotations}. For $e^{i\theta}\in U(1)$, the rotation fixing $q_0=(0,0,0)$ is given by \begin{equation*}R_{\theta}:(\zeta,t)\mapsto(e^{i\theta}\zeta,t).\end{equation*} All other Heisenberg rotations may be obtained from these by conjugating by a Heisenberg translation.

For $\lambda\in\mathbb{R}_+$, {\it Heisenberg dilation} by $\lambda$ fixing $q_\infty$ and $q_0=(0,0,0)\in\partial\textbf{H}^2_{\mathbb{C}}$ is given by
\begin{equation*}D_\lambda:(\zeta,t)\mapsto(\lambda \zeta,\lambda^2 t).\end{equation*}
All other Heisenberg dilations fixing $q_\infty$ may be obtained by conjugating by a Heisenberg translation.

The stabiliser of $q_{\infty}$ in $PU(2,1)$ is generated by all Heisenberg translations, rotations and dilations. However, only Heisenberg translations and rotations are isometric with respect to various natural metrics on $\mathfrak{R}$. For this reason the group generated by all Heisenberg translations and rotations, which is the semidirect product $U(1)\ltimes\mathfrak{R}$, is called the {\it Heisenberg isometry group} $Isom(\mathfrak{R})$. The nontrivial central elements of the Heisenberg isometry group are precisely the vertical translations. In particular, each element of $Isom(\mathfrak{R})$ preserves every horosphere.

There is a canonical projection from $\mathfrak{R}$ to $\mathbb{C}$ called {\it vertical projection} and denoted by $\Pi$, given by $\Pi: (\zeta,t)\longmapsto \zeta.$
Using the exact sequence
\begin{equation*}0\longrightarrow\mathbb{R}\longrightarrow\mathfrak{R}\stackrel{\Pi}{\longrightarrow}\mathbb{C}\longrightarrow0,\end{equation*}
we obtain the exact sequence (see Scott \cite{Sc} page 467)
\begin{equation}\label{eq:2-2}0\longrightarrow\mathbb{R}\longrightarrow  Isom(\mathfrak{R})\stackrel{\Pi_{*}}{\longrightarrow}  Isom(\mathbb{C})\longrightarrow1.\end{equation}
Here $Isom(\mathbb{C})$ is the group of orientation preserving
Euclidean isometries of $\mathbb{C}$.

Observe the elements in $Isom(\mathbb{C})$ can be represented by
matrices in $GL(2,\mathbb{C})$ of the form
\begin{equation*}\left[\begin{array}{cc}
 e^{i\theta}&\zeta_0\\0&1
\end{array}\right]
\left[\begin{array}{c}
 \zeta\\1
 \end{array}\right]=\left[\begin{array}{c}
 e^{i\theta}\zeta+\zeta_0\\1
\end{array}\right]\end{equation*}
Therefore, the map $\Pi_{*}$ can be given by
\begin{equation}\label{eq:2-3}\Pi_{*}:\left[ \begin{array}{ccc}
 1&-\bar{\zeta_{0}}e^{i\theta}&(-|\zeta_{0}|^{2}+it_{0})/2\\0&e^{i\theta}&\zeta_{0}\\0&0&1
\end{array}\right]\longrightarrow\left[ \begin{array}{cc}
 e^{i\theta}&\zeta_0\\0&1
\end{array}\right].\end{equation}
It is clear that
\begin{equation*}Ker(\Pi_{*})=\left\lbrace \left[ \begin{array}{ccc}
 1&0&it_{0}/2\\0&1&0\\0&0&1
\end{array}\right]:t_{0}\in\mathbb{R}\right\rbrace\end{equation*}
is the group of vertical translations fixing $q_{\infty}$.

\medskip

\subsection{Isometric spheres}

\medskip

Given an element $G\in PU(2,1)$ with satisfying $G(q_{\infty})\neq
q_{\infty}$, we define the isometric sphere of $G$ to be the
hypersurface
$$\left\{\mathbf{z}\in\textbf{H}^2_{\mathbb{C}}:|\langle \mathbf{z},q_{\infty}\rangle|=|\langle \mathbf{z},G^{-1}(q_{\infty})\rangle|\right\}.$$
For example, the isometric sphere of
$$I_0=\left[\begin{array}{ccc}
0&0&1\\
0&-1&0\\
1&0&0
\end{array}\right]
$$
is \begin{equation}\label{eq:2-4}
\mathcal{B}_0=\left\{(\zeta,t,u)\in\mathfrak{S}:\left||\zeta|^2+u+it\right|=2\right\}\end{equation} in horospherical
coordinates or
\begin{equation}\label{eq:2-5}\mathcal
{B}_0=\left\{[z_1,z_2,z_3]\in\textbf{H}^2_{\mathbb{C}}:|z_1|=|z_3|\right\}\end{equation}
in homogeneous coordinates.

All other isometric spheres are images of $\mathcal
{B}_0$
by Heisenberg dilations, rotations and translations. Thus the
isometric sphere with radius $r$ and centre $(\zeta_0,t_0,0)$ is given
by
$$\left\{(\zeta,t,u):\left||\zeta-\zeta_0|^2+u+it-it_0+2i\Im m(\zeta\bar{\zeta}_0)\right|=r^2\right\}.$$

If $G$ has the matrix form \begin{equation}\label{eq:2-6}\left[\begin{array}{ccc}
a&b&c\\
d&e&f\\
g&h&j
\end{array}\right]
,\end{equation} then $G(q_{\infty})\neq q_{\infty}$ if and
only if $g\neq0$. The isometric sphere of $G$ has radius
$r=\sqrt{2/|g|}$ and centre $G^{-1}(q_{\infty})$, which in
horospherical coordinates is
$$(\zeta_0,t_0,0)=(\bar{h}/\bar{g},2\Im m(\bar{j}/\bar{g}),0).$$

Isometric spheres are examples of bisectors. Mostow \cite{Mo} showed that
a bisector is the preimage of a geodesic, called spine, under
orthogonal projection onto the unique complex line containing it.
The fibres of this projection are complex lines called the slices of
the bisector. Goldman \cite{Go} showed that a bisector is the union of
all totally real Larangian planes containing the spine. Such
Lagrangian planes are called the meridians.

\section{On the structure of the stabiliser}

In this section we will obtain the generators and relations of the stabiliser of Picard modular groups by analysis of the fundamental domain in Heisenberg group.

Let $\mathcal{O}_d$ be the ring of integers in the quadratic imaginary number field $\mathbb{Q}(i\sqrt{d})$, where $d$ is
a positive square-free integer. If $d\equiv1,2\ (mod\ 4)$, then $\mathcal{O}_d=\mathbb{Z}[i\sqrt{d}]$ and if
 $d\equiv3\ (mod\ 4)$, then $\mathcal{O}_d=\mathbb{Z}[\omega_d]$, where $\omega_d=(1+i\sqrt{d})/2   $. The group $\Gamma_d=PU(2,1;\mathcal{O}_d)$
 is called {\it Euclidean Picard modular group} if the ring $\mathcal{O}_d$ is Euclidean, namely, only the rings $\mathcal{O}_1,\mathcal{O}_2,\mathcal{O}_3,\mathcal{O}_7,\mathcal{O}_{11}$. Further relative to amalgamation property, these five groups can be subclassified into three groupings $\{\Gamma_1\},$ $\{\Gamma_3\},$ $\{\Gamma_2,\Gamma_7,\Gamma_{11}\}$. Since two classes $\{\Gamma_1\},\{\Gamma_3\}$ (\textit{c.f.} \cite{FP}, \cite{FFP}) have been studied in detail, we mainly describe the remaining class $\{\Gamma_2,\Gamma_7,\Gamma_{11}\}$.

\subsection{The stabiliser of $q_\infty$}

First we want to analyse $(\Gamma_d)_\infty$ with $d=2,7,11$, the stabiliser of $q_\infty$. Every element of $(\Gamma_d)_\infty$ is upper triangular and its diagonal entries are units in $\mathcal{O}_d$. Recall that the units of $\mathcal{O}_1$ are $\pm1,\pm i$, they are $\pm1,\pm\omega,\pm\omega^2$ for $\mathcal{O}_3$ and they are $\pm1$ for others. Therefore $(\Gamma_d)_\infty$ contains no dilations and so is a subgroup of $Isom(\mathfrak{R})$ and fits into the exact sequence as
$$0\longrightarrow\mathbb{R}\cap(\Gamma_d)_\infty\longrightarrow
(\Gamma_d)_\infty\stackrel{\Pi_{*}}{\longrightarrow}\Pi_{*}((\Gamma_d)_\infty)\longrightarrow1.$$

We can write the isometry group of the integer lattice as
\begin{equation*}
Isom(\mathcal{O}_d)=\left\{\left[\begin{array}{cc}\alpha&\beta\\0&1\end{array}\right]:\alpha,\beta\in\mathcal{O}_d, \alpha\ \text{is a unit}\right\}.
\end{equation*}

We now find the image and kernel in this exact sequence.

\begin{proposition}\label{eq:3-1}The stabiliser
$(\Gamma_d)_\infty$ of $q_{\infty}$ in $\Gamma_d$ satisfies
$$0\longrightarrow 2\sqrt{d}\mathbb{Z}\longrightarrow(\Gamma_d)_\infty\stackrel{\Pi_{*}}{\longrightarrow}\Delta\longrightarrow1,$$
where $\Delta\subset Isom(\mathcal{O}_d)$ is of index 2 if $d\equiv2(mod\ 4)$ and $\Delta=Isom(\mathcal{O}_d)$ if $d\equiv3(mod\ 4)$.
\end{proposition}

\begin{proof} Although we only consider the cases $d=2,7,11$, the ring $\mathcal{O}_2$ represents those for the values of $d$ with $d\equiv2(mod\ 4)$ and the rings $\mathcal{O}_7, \mathcal{O}_{11}$ represent those of the values $d\equiv3(mod\ 4)$, the remaining case is the same as $\mathcal{O}_1$ which has been done in \cite{FFP}. Observe that $Isom(\mathcal{O}_d)$ is generated by the subgroup of translations
$$\left\{\hat{T}_\beta=\left[\begin{array}{cc}1&\beta\\0&1\end{array}\right]:\beta\in\mathcal{O}_d\right\}$$
and the finite subgroup of order two
$$\left\{\hat{R}_\alpha=\left[\begin{array}{cc}\alpha&0\\0&1\end{array}\right]:\alpha\in\mathcal{O}_d,\text{$\alpha$ is a unit}\right\}.$$
Then, to understand $\Delta\subset Isom(\mathcal{O}_d)$, it suffices to determine which translations can be lifted. We divide into two cases to complete the proof.

\medskip

(i)\ {\bf The case $\mathcal{O}_d$ with $d\equiv2(mod\ 4)$}

 Suppose that $\beta\in\mathcal{O}_d=\mathbb{Z}[i\sqrt{d}]$ and consider the translation $\hat{T}_\beta$ by $\beta$ in $\mathbb{Z}[i\sqrt{d}]$ given above. The preimage of $\hat{T}_\beta$ under $\Pi_*$ has the
form
$$T_{\beta,t}=\left[\begin{array}{ccc}1&-\bar{\beta}&\frac{-|\beta|^2+it}{2}\\0&1&\beta\\0&0&1\end{array}\right].$$
This map is in $PU(2,1;\mathbb{Z}[i\sqrt{d}])$ if and only if $|\beta|^2$ is an even integer and $t\in2\sqrt{d}\mathbb{Z}$ . Writing $\beta=m+i\sqrt{d}n$ for $m,n\in\mathbb{Z}$, then we can obtain $m\equiv0\ (mod\ 2)$ from the conditions $|\beta|^2=m^2+dn^2\in2\mathbb{Z}$ and  $d\equiv2(mod\ 4)$. Therefore, we conclude that $\Delta\subset Isom(\mathbb{Z}[i\sqrt{d}])$ is of index 2. Also, the kernel of $\Pi_*$ is generated by
$$\left[\begin{array}{ccc}1&0&i\sqrt{d}\\0&1&0\\0&0&1\end{array}\right],$$ which is a vertical translation of $(0,2\sqrt{2})$.

\medskip

(ii)\ {\bf The case $\mathcal{O}_d$ with $d\equiv3(mod\ 4)$}

Suppose that $\beta=m+n\frac{1+i\sqrt{d}}{2}\in\mathcal{O}_d$ with $m,n\in\mathbb{Z}$ for $d\equiv3(mod\ 4)$. By the same argument of (i), it only suffices to determine $m,n$ such that $|\beta|^2$ is an integer. For $d\equiv3(mod\ 4)$, it is easy to show that $|\beta|^2=m^2+mn+n^2(d+1)/4\in\mathbb{Z}$ for any $m,n\in\mathbb{Z}$, which implies that $\Delta=Isom(\mathcal{O}_d)$. Obviously, the kernel of  $\Pi_*$ is generated by a vertical translation of $(0,2\sqrt{d})$.

\end{proof}

\subsection{Fundamental domain for the stabiliser}
As the first step toward the construction of a fundamental domain for the action of $(\Gamma_d)_\infty$ on $\mathfrak{R}$ for $d=2,7,11$, we shall find the suitable generators of $Isom(\mathcal{O}_d)$ to construct a fundamental domain in $\mathbb{C}$.

In the proof of Proposition \ref{eq:3-1} we saw that $\Delta=\Pi_*((\Gamma_2)_\infty)$ is a subgroup of index 2 in $Isom(\mathcal{O}_2)$ consisting of elements of $GL(2,\mathcal{O}_2)$ of the form
$$\left\{\left[\begin{array}{cc}(-1)^j&m+i\sqrt{2}n\\0&1\end{array}\right]:j=0,1,m,n\in\mathbb{Z},m\equiv0(mod\ 2)\right\}.$$
A fundamental domain for this group is the triangle in  $\mathbb{C}$ with vertices at $-1+\sqrt{2}i/2$ and $1\pm\sqrt{2}i/2$; see (a) in Figure 3.1.
Side paring maps are given by
$$r^{(2)}_1=\left[\begin{array}{cc}-1&0\\0&1\end{array}\right],\ r^{(2)}_2=\left[\begin{array}{cc}-1&2\\0&1\end{array}\right],\
r^{(2)}_3=\left[\begin{array}{cc}-1&\sqrt{2}i\\0&1\end{array}\right].
$$
The first of these is a rotation of order 2 fixing origin, the second is a rotation of order 2 fixing $1/2$ and the third is a rotation of order 2 fixing $\sqrt{2}i/2$. Indeed every element of $\Delta=GL(2,\mathcal{O}_2)$ is generated by $r^{(2)}_1,r^{(2)}_2,r^{(2)}_3$ as follows
\begin{align*}\left[\begin{array}{cc}(-1)^j&2m+\sqrt{2}ni\\0&1\end{array}\right]&=\left[\begin{array}{cc}1&2\\0&1\end{array}\right]^m
\left[\begin{array}{cc}1&\sqrt{2}i\\0&1\end{array}\right]^n\left[\begin{array}{cc}-1&0\\0&1\end{array}\right]^j\\
&=\left(r^{(2)}_2r^{(2)}_1\right)^m\left(r^{(2)}_3r^{(2)}_1\right)^n\left(r^{(2)}_1\right)^j.
\end{align*}

As the same argument, a fundamental domain for $Isom(\mathcal{O}_d)$ with $d=7$ or $11$ is the triangle in  $\mathbb{C}$ with vertices at $(-1+i\sqrt{d})/4$, $(1-i\sqrt{d})/4$ and $(3+i\sqrt{d})/4$; see (b) in Figure 3.1. Side paring maps are given by
$$r^{(d)}_1=\left[\begin{array}{cc}-1&0\\0&1\end{array}\right],\ r^{(d)}_2=\left[\begin{array}{cc}-1&1\\0&1\end{array}\right],\
r^{(d)}_3=\left[\begin{array}{cc}-1&(1+i\sqrt{d})/2\\0&1\end{array}\right].
$$
All these maps are rotations by $\pi$ fixing $0,1/2$ and $(1+i\sqrt{d})/4$ respectively.

\begin{figure}
\centering
  \subfigure[]{
  \psfrag{r1}{\Small$r_1^{(2)}$}\psfrag{r2}{\Small$r_2^{(2)}$}
\psfrag{r3}{\Small$r_3^{(2)}$}
\psfrag{0}{\Small$0$}
\psfrag{1}{\Small$1$}\psfrag{2}{\Small$\frac{i\sqrt{2}}{2}$}
    \label{fig:mini:subfig:1}
 \begin{minipage}[1]{0.45\textwidth} \centering
      \includegraphics[width=1.9in]{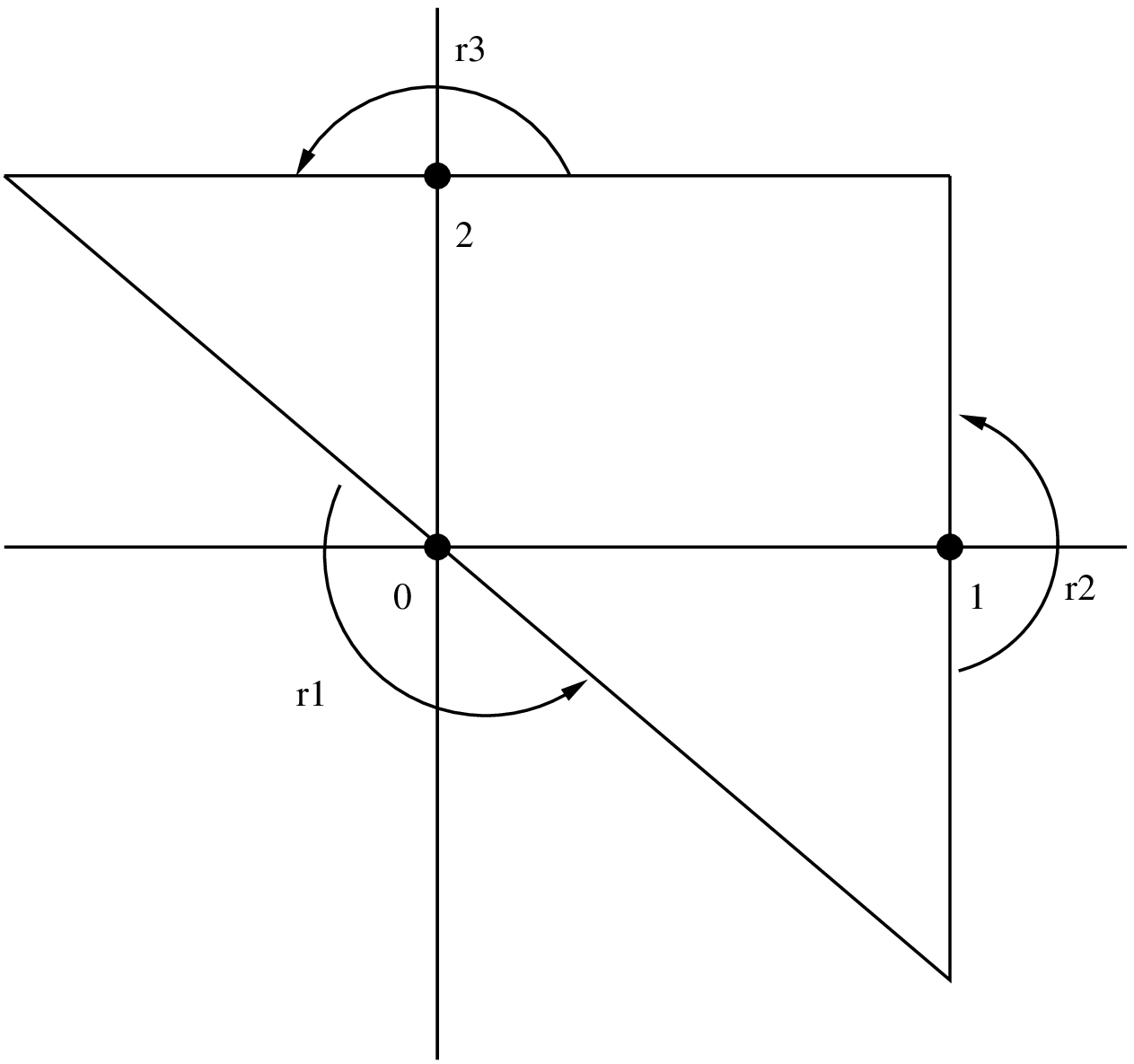}
  \end{minipage}
}%
  \subfigure[]{
  \psfrag{r1}{\Small$r_1^{(d)}$}\psfrag{r2}{\Small$r_2^{(d)}$}
\psfrag{r3}{\Small$r_3^{(d)}$}
\psfrag{0}{\Small$0$}\psfrag{wd}{\Small$\frac{\omega_d}{2}$}
\psfrag{1/2}{\Small$\frac{1}{2}$}
    \label{fig:mini:subfig:2} 
    \begin{minipage}[3]{0.45\textwidth}
      \centering
      \includegraphics[width=1.8in]{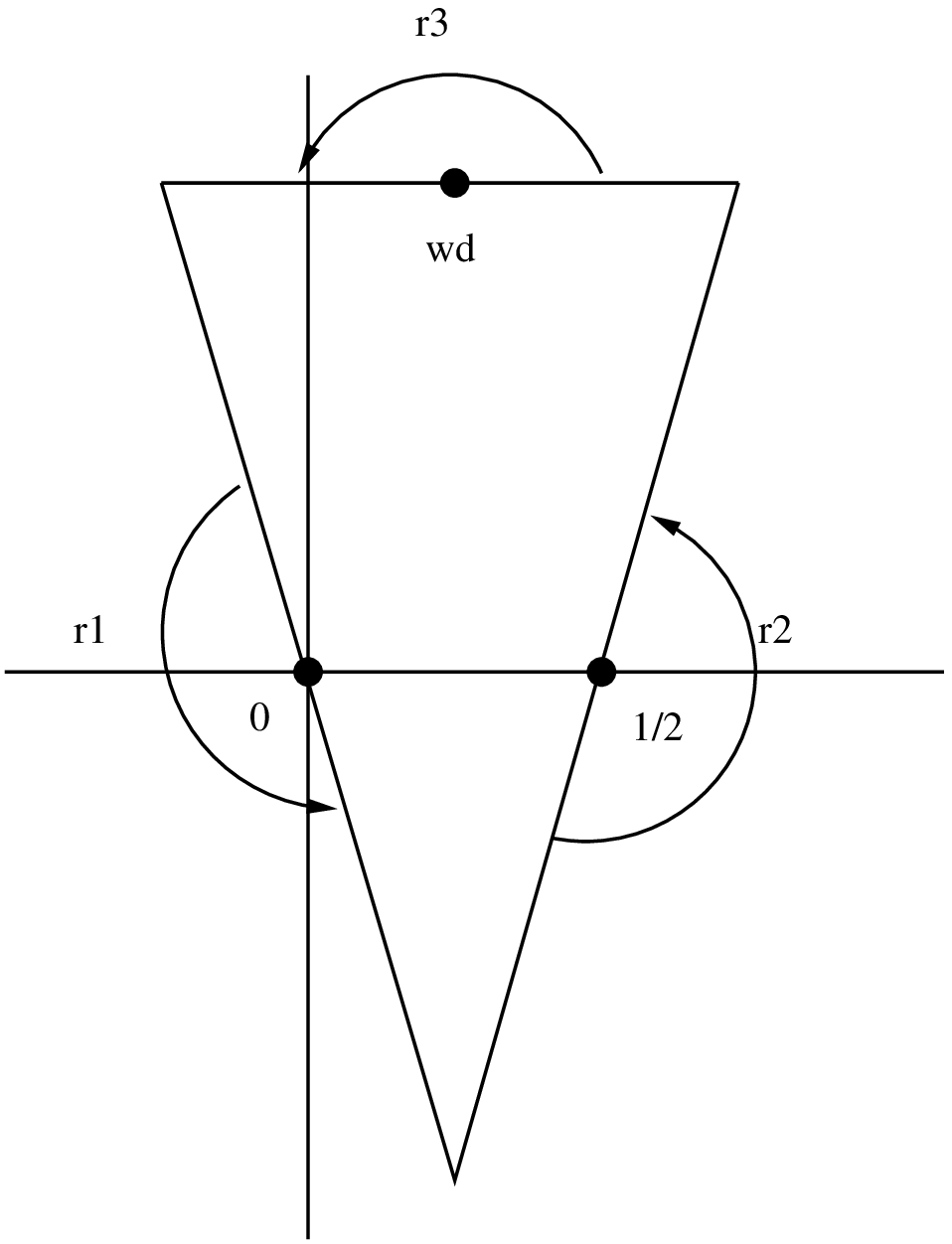}
    \end{minipage}
    }
\medskip
\makeatletter\def\@captype{figure}\makeatother
\bigskip
\caption{(a)\ Fundamental domain for a subgroup $\Delta$ of $Isom(\mathcal{O}_2)$ with index 2.
(b)\ Fundamental domain for $Isom(\mathcal{O}_d)$ with $d=7,11$. This is also true for all the values of $d$ with $d\equiv3(mod\ 4)$.}
  \label{fig:mini:subfig} 
\end{figure}

\medskip

In order to produce a fundamental domain for $(\Gamma_d)_\infty$ we look at all the preimages of
the triangle (that is a fundamental domain of $\Pi_*((\Gamma_d)_\infty)$)  under vertical projection
$\Pi$ and we intersect this with a fundamental domain for $ker(\Pi_*)$. The inverse of image of the
triangle under $\Pi$ is an infinite prism. The kernel of $\Pi_*$ is the infinite cyclic group
generated by $T$, the vertical translation by $(0,2\sqrt{d})$.

\begin{figure}
\centering
\psfrag{T}{\Large$T^{(2)}$}\psfrag{R1}{\Large$R_1^{(2)}$}\psfrag{R2}{\Large$R_2^{(2)}$}\psfrag{R3}{\Large$R_3^{(2)}$}
\psfrag{TR2}{\Large$T^{(2)}R_2^{(2)}$}\psfrag{TR3}{\Large$T^{(2)}R_3^{(2)}$}
\psfrag{z1+}{\Large$v_1^+$}\psfrag{z2+}{\Large$v_2^+$}\psfrag{z3+}{\Large$v_3^+$}\psfrag{z4+}{\Large$v_4^+$}
\psfrag{z5+}{\Large$v_5^+$}\psfrag{z6+}{\Large$v_6^+$}
\psfrag{z1-}{\Large$v_1^-$}\psfrag{z2-}{\Large$v_2^-$}\psfrag{z3-}{\Large$v_3^-$}\psfrag{z4-}{\Large$v_4^-$}\psfrag{z5-}{\Large$v_5^-$}
\psfrag{z6-}{\Large$v_6^-$}
\resizebox{2.5in}{!}{\includegraphics{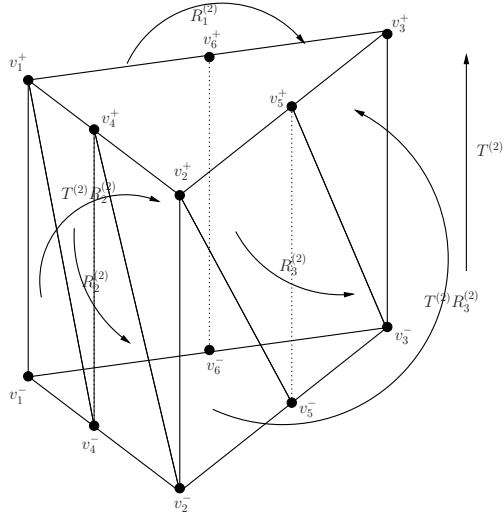}}
\medskip
\makeatletter\def\@captype{figure}\makeatother
\caption{A fundamental domain $\mathbf{\Sigma}_2$ for $(\Gamma_2)_\infty$ in the Heisenberg group: the map $R^{(2)}_1$
 rotates through $\pi$ about $\zeta=0$, the map $R^{(2)}_2$ is a Heisenberg rotation through $\pi$
  about $\zeta=1$ and the map $R^{(2)}_3$ is a Heisenberg rotation through $\pi$ about $\zeta=\sqrt{2}i/2$.}
\end{figure}

\medskip

\begin{proposition} $(\Gamma_2)_\infty$ is generated by $$ R^{(2)}_1=\left[\begin{array}{ccc}1&0&0\\0&-1&0\\0&0&1\end{array}\right],
 R^{(2)}_2=\left[\begin{array}{ccc}1&2&-2\\0&-1&2\\0&0&1\end{array}\right],
 R^{(2)}_3=\left[\begin{array}{ccc}1&-i\sqrt{2}&-1\\0&-1&i\sqrt{2}\\0&0&1\end{array}\right]$$
and $$T^{(2)}=\left[\begin{array}{ccc}1&0&i\sqrt{2}\\0&1&0\\0&0&1\end{array}\right].$$

A presentation is given by
$$(\Gamma_2)_{\infty}=\langle R^{(2)}_j,T^{(2)}|{R^{(2)}_j}^2=[T^{(2)},R^{(2)}_j]=
\left({T^{(2)}}^2R^{(2)}_1R^{(2)}_3R^{(2)}_2\right)^2=id\rangle.
$$
\end{proposition}

\begin{proof}
Those matrices are constructed by lifting generators of the subgroup $\Delta\subset Isom(\mathcal{O}_2)$ with index 2
and also $T^{(2)}$ is a generator of the kernel of the map $\Pi_*$. A fundamental domain can be constructed with
side pairings as Figure 3.2, where the vertices of the prism are
$v_3^+=(-1+\sqrt{2}i/2,\sqrt{2}),$ $v_2^+=(1+\sqrt{2}i/2,\sqrt{2}),$ $v_1^+=(1-\sqrt{2}i/2,\sqrt{2})$ for the upper cap of the prism
and $v_3^-=(-1+\sqrt{2}i/2,-\sqrt{2}),$ $v_2^-=(1+\sqrt{2}i/2,-\sqrt{2}),$ $v_1^-=(1-\sqrt{2}i/2,-\sqrt{2})$ for the base. In particular, the points
$v_4^{\pm},$ $v_5^{\pm},$ $v_6^{\pm}$ are the middle points of the edges $(v_1^{\pm},v_2^{\pm}),$ $(v_2^{\pm},v_3^{\pm})$ and $(v_3^{\pm},v_1^{\pm})$, respectively.

The actions of side-pairing maps on $\mathfrak{R}$ are given by
\begin{align*}
R^{(2)}_1(\zeta,t)&=(-\zeta,t),\\
R^{(2)}_2(\zeta,t)&=(-\zeta+2,t+4\Im m{\zeta}),\\
R^{(2)}_3(\zeta,t)&=(-\zeta+i\sqrt{2},t-2\sqrt{2}\Re e{\zeta}),\\
T^{(2)}(\zeta,t)&=(\zeta,t+2\sqrt{2}).
\end{align*}
We describe the side pairing in terms of the action on the vertice:
\begin{eqnarray*}
R^{(2)}_1&:&(v_6^+,v_1^+,v_1^-,v_6^-)\longrightarrow(v_6^+,v_3^+,v_3^-,v_6^-),\\
R^{(2)}_2&:&(v_1^+,v_4^+,v_4^-)\longrightarrow(v_2^-,v_4^+,v_4^-),\\
T^{(2)}R^{(2)}_2&:&(v_1^+,v_1^-,v_4^-)\longrightarrow(v_2^+,v_2^-,v_4^+),\\
R^{(2)}_3&:&(v_2^+,v_5^+,v_5^-)\longrightarrow(v_3^-,v_5^+,v_5^-),\\
T^{(2)}R^{(2)}_3&:&(v_2^+,v_2^-,v_5^-)\longrightarrow(v_3^+,v_3^+,v_5^+),\\
T^{(2)}&:&(v_1^-,v_4^-,v_2^-,v_5^-,v_3^-,v_6^-)\longrightarrow(v_1^+,v_4^+,v_2^+,v_5^+,v_3^+,v_6^+).
\end{eqnarray*}
The presentation can be obtained following from the edge cycles of the fundamental domain.
\end{proof}

\medskip

\begin{proposition} $(\Gamma_7)_\infty$ is generated by $$ R^{(7)}_1=\left[\begin{array}{ccc}1&0&0\\0&-1&0\\0&0&1\end{array}\right],
 R^{(7)}_2=\left[\begin{array}{ccc}1&1&-\bar{\omega}_7\\0&-1&1\\0&0&1\end{array}\right],
 R^{(7)}_3=\left[\begin{array}{ccc}1&\bar{\omega}_7&-1\\0&-1&\omega_7\\0&0&1\end{array}\right]$$
and $$T^{(7)}=\left[\begin{array}{ccc}1&0&i\sqrt{7}\\0&1&0\\0&0&1\end{array}\right].$$

A presentation is given by
\begin{align*}(\Gamma_7)_{\infty}&=\langle R^{(7)}_j,T^{(7)}|{R^{(7)}_1}^2={R^{(7)}_3}^2=[T^{(7)},R^{(7)}_1]=[T^{(7)},R^{(7)}_3]\\
&\hspace{4cm}=T^{(7)}{R^{(7)}_2}^{-2}=\left(R^{(7)}_1R^{(7)}_3R^{(7)}_2\right)^2=id\rangle.
\end{align*}
\end{proposition}

\begin{proof}
Those matrices are constructed by lifting generators of $Isom(\mathcal{O}_7)$ and also $T^{(7)}$ is a generator of the kernel of the map $\Pi_*$.
A fundamental domain can be constructed with
side pairings as Figure 3.3, where the vertices of the prism are
$v_1^+=((1-i\sqrt{7})/4,\sqrt{7}),$ $v_2^+=((3+i\sqrt{7})/4,\sqrt{7}),$ $v_4^+=((-1+i\sqrt{7})/4,\sqrt{7})$ for the upper cap of the prism
and $v_1^-=((1-i\sqrt{7})/4,-\sqrt{7}),$ $v_2^-=((3+i\sqrt{7})/4,-\sqrt{7}),$ $v_4^-=((-1+i\sqrt{7})/4,-\sqrt{7})$ for the base.
The points $v_3^{\pm},v_5^{\pm}$ are the middle points of the edges $(v_2^{\pm},v^{\pm}_4)$ and $(v_4^{\pm},v^{\pm}_1)$. In particular,
we introduce more three points $w_1^+=((1-i\sqrt{7})/4,\sqrt{7}/2),$ $w_2^-=((3+i\sqrt{7})/4,-\sqrt{7}/2)$ and $w_3^+=((-1+i\sqrt{7})/4,\sqrt{7}/2)$.
The actions of side-pairing maps on $\mathfrak{R}$ are given by
\begin{align*}
R^{(7)}_1(\zeta,t)&=(-\zeta,t),\\
R^{(7)}_2(\zeta,t)&=(-\zeta+1,t+2\Im m{\zeta}+\sqrt{7}),\\
R^{(7)}_3(\zeta,t)&=(-\zeta+\omega_7,t+2\Im m{(\bar{\omega}_7\zeta})),\\
T^{(7)}(\zeta,t)&=(\zeta,t+2\sqrt{7}).
\end{align*}
We describe the side pairing in terms of the action on the vertice:
\begin{eqnarray*}
R^{(7)}_1&:&(v_5^+,v_1^+,v_1^-,v_5^-)\longrightarrow(v_5^+,v_4^+,v_4^-,v_5^-),\\
R^{(7)}_2&:&(v_1^-,v_2^-,w_1^-,w_1^+)\longrightarrow(w_1^-,w_1^+,v_1^+,v_2^+),\\
R^{(7)}_3&:&(v_2^+,w_1^-,v_3^-,v_3^+)\longrightarrow(w_2^+,v_4^-,v_3^-,v^+_3),\\
T^{(7)}R^{(7)}_3&:&(w_1^-,v_2^-,v_3^-)\longrightarrow(v_4^+,w_2^+,v_3^+),\\
T^{(7)}&:&(v_1^-,v_2^-,v_3^-,v_4^-,v_5^-)\longrightarrow(v_1^+,v_2^+,v_3^+,v_4^+,v_5^+).
\end{eqnarray*}
The presentation can be obtained following from the edge cycles of the fundamental domain.
\end{proof}

\begin{figure}
\centering
\psfrag{T}{\Large$T^{(7)}$}\psfrag{R1}{\Large$R_1^{(7)}$}\psfrag{R2}{\Large$R_2^{(7)}$}\psfrag{R3}{\Large$R_3^{(7)}$}
\psfrag{TR3}{\Large$T^{(7)}R_3^{(7)}$}
\psfrag{z1+}{\Large$v_1^+$}\psfrag{z2+}{\Large$v_2^+$}\psfrag{z3+}{\Large$v_3^+$}\psfrag{z4+}{\Large$v_4^+$}
\psfrag{z5+}{\Large$v_5^+$}
\psfrag{z1-}{\Large$v_1^-$}\psfrag{z2-}{\Large$v_2^-$}\psfrag{z3-}{\Large$v_3^-$}
\psfrag{z4-}{\Large$v_4^-$}\psfrag{z5-}{\Large$v_5^-$}
\psfrag{w1+}{\Large$w_1^+$}\psfrag{w2+}{\Large$w_2^+$}\psfrag{w1-}{\Large$w_1^-$}
\resizebox{2.5in}{!}{\includegraphics{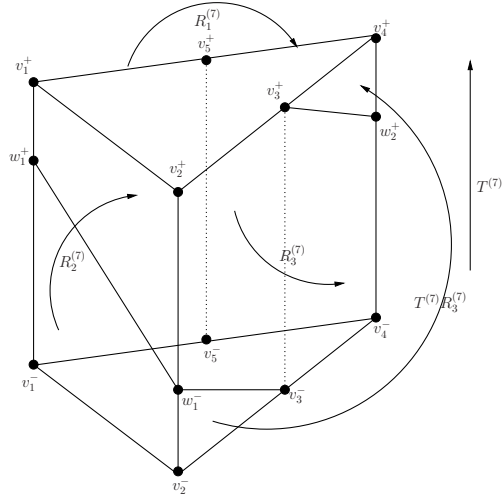}}
\medskip
\makeatletter\def\@captype{figure}\makeatother
\caption{A fundamental domain $\mathbf{\Sigma}_7$ for $(\Gamma_7)_\infty$ in the Heisenberg group: the map $R^{(7)}_1$ rotates through $\pi$ about $\zeta=0$,
the action of parabolic $R_2^{(7)}$ is a Heisenberg rotation through $\pi$ about $\zeta=1/2$ followed by an upward vertical
translation by $\sqrt{7}$ and the map $R_3^{(7)}$ is a Heisenberg rotation through $\pi$ about $\zeta=(1+i\sqrt{7})/4$.}
\end{figure}

\medskip

\begin{proposition} $(\Gamma_{11})_\infty$ is generated by $$ R^{(11)}_1=\left[\begin{array}{ccc}1&0&0\\0&-1&0\\0&0&1\end{array}\right],
 R^{(11)}_2=\left[\begin{array}{ccc}1&1&-\bar{\omega}_{11}\\0&-1&1\\0&0&1\end{array}\right],
 $$
$$R^{(11)}_3=\left[\begin{array}{ccc}1&\bar{\omega}_{11}&-1-\bar{\omega}_{11}\\0&-1&\omega_{11}\\0&0&1\end{array}\right]\quad \text{and}\quad T^{(11)}=\left[\begin{array}{ccc}1&0&i\sqrt{11}\\0&1&0\\0&0&1\end{array}\right].$$

A presentation is given by
\begin{align*}(\Gamma_{11})_{\infty}&=\langle R^{(11)}_j,T^{(11)}|{R^{(11)}_1}^2=[T^{(11)},R^{(11)}_1]=T^{(11)}{R^{(11)}_2}^{-2}\\
&\hspace{2cm}=T^{(11)}{R^{(11)}_3}^{-2}=T^{(11)}\left(R^{(11)}_1R^{(11)}_3R^{(11)}_2\right)^{-2}=id\rangle.
\end{align*}
\end{proposition}

\begin{proof}
Those matrices are constructed by lifting generators of $Isom(\mathcal{O}_{11})$ and also $T^{(11)}$ is a generator of the kernel of the map $\Pi_*$.
A fundamental domain can be constructed with
side pairings as Figure 3.4, where the vertices of the prism are
$v_1^+=((1-i\sqrt{11})/4,\sqrt{11}),$ $v_2^+=((3+i\sqrt{11})/4,3\sqrt{11}/2),$ $v_3^+=((-1+i\sqrt{11})/4,2\sqrt{11})$ for the upper cap of the prism
and $v_1^-=((1-i\sqrt{11})/4,-\sqrt{11}),$ $v_2^-=((3+i\sqrt{11})/4,-\sqrt{11}/2),$ $v_3^-=((-1+i\sqrt{11})/4,0)$ for the base.
The points $v_0^{\pm}$ are the middle points of the edges $(v_1^{\pm},v^{\pm}_3)$. In particular, we introduce more three points $w_1=((1-i\sqrt{11})/4,0),$
 $w_2=((3+i\sqrt{11})/4,\sqrt{11}/2)$ and $w_3=((-1+i\sqrt{11})/4,\sqrt{11})$.
The actions of side-pairing maps on $\mathfrak{R}$ are given by
\begin{align*}
R^{(11)}_1(\zeta,t)&=(-\zeta,t),\\
R^{(11)}_2(\zeta,t)&=(-\zeta+1,t+2\Im m{\zeta}+\sqrt{11}),\\
R^{(11)}_3(\zeta,t)&=(-\zeta+\omega_{11},t+2\Im m{(\bar{\omega}_{11}\zeta})+\sqrt{11}),\\
T^{(11)}(\zeta,t)&=(\zeta,t+2\sqrt{11}).
\end{align*}
We describe the side pairing in terms of the action on the vertice:
\begin{eqnarray*}
R^{(11)}_1&:&(v_0^+,v_1^+,w_1,v_0^-)\longrightarrow(v_0^+,w_3,v_3^-,v_0^-),\\
T^{(11)}R^{(11)}_1&:&(w_1,v_1^-,v_0^-)\longrightarrow(v_3^+,w_3,v_0^+),\\
R^{(11)}_2&:&(v_1^+,w_1,v_1^-,v_2^-)\longrightarrow(v_2^+,w_2,v_2^-,v_1^+),\\
R^{(11)}_3&:&(v_2^+,w_2,v_2^-,v_3^-)\longrightarrow(v_3^+,w_3,v_3^-,v_2^+),\\
T^{(11)}&:&(v_0^-,v_1^-,v_2^-,v_3^-)\longrightarrow(v_0^+,v_1^+,v_2^+,v_3^+).
\end{eqnarray*}
The presentation can be obtained following from the edge cycles of the fundamental domain.
\end{proof}

\begin{figure}
\centering
\psfrag{T}{\Large$T^{(11)}$}\psfrag{R1}{\Large$R_1^{(11)}$}\psfrag{R2}{\Large$R_2^{(11)}$}\psfrag{R3}{\Large$R_3^{(11)}$}
\psfrag{TR1}{\Large$T^{(11)}R_1^{(11)}$}
\psfrag{z1+}{\Large$v_1^+$}\psfrag{z2+}{\Large$v_2^+$}\psfrag{z3+}{\Large$v_3^+$}\psfrag{z0+}{\Large$v_0^+$}
\psfrag{z1-}{\Large$v_1^-$}\psfrag{z2-}{\Large$v_2^-$}\psfrag{z3-}{\Large$v_3^-$}
\psfrag{z0-}{\Large$v_0^-$}
\psfrag{w1}{\Large$w_1$}\psfrag{w2}{\Large$w_2$}\psfrag{w3}{\Large$w_3$}
\resizebox{2.5in}{!}{\includegraphics{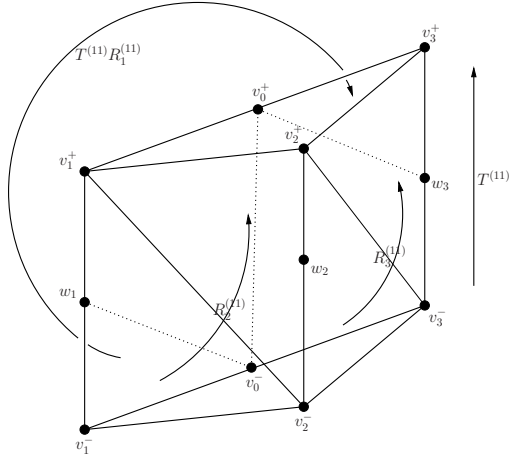}}
\medskip
\makeatletter\def\@captype{figure}\makeatother
\caption{A fundamental domain $\mathbf{\Sigma}_{11}$ for $(\Gamma_{11})_\infty$ in the Heisenberg group: the map $R^{(11)}_1$ rotates through $\pi$ about $\zeta=0$, the action of parabolic $R_2^{(11)}$ is a screw Heisenberg rotation through $\pi$ about $\zeta=1/2$ followed by an upward vertical translation by $\sqrt{11}$ and the map $R_3^{(11)}$ is a screw Heisenberg rotation through $\pi$ about $\zeta=(1+i\sqrt{11})/4$ followed by an upward vertical translation by $\sqrt{11}$.}
\end{figure}

\section{The statement of our method and results}

In this section, we introduce the method used in (\cite{FP}, \cite{FFP}) to determine the generators of the Euclidean Picard groups and then state our results.

Recall that the map
$$I_0=\left[\begin{array}{ccc}0&0&1\\0&-1&0\\1&0&0\end{array}\right],$$
defined in the Section 2.3. We consider the isometric sphere $\mathcal{B}_0$ of $I_0$ given by (\ref{eq:2-4}), which is a Cygan sphere centred $o=(0,0,0)$ with radius $\sqrt{2}$. Observe that $I_0$ maps $\mathcal{B}_0$ to itself and swaps the inside and the outside of $\mathcal{B}_0$. Given an element of $\Gamma_d$ as the form (\ref{eq:2-6}), we know that the radius of isometric sphere is $\sqrt{2/|g|}$. For each case $\mathcal{O}_d$, the radius of isometric sphere is not greater than $\sqrt{2}$ since the absolute of $g$ is not smaller than 1 for $g\in\mathcal{O}_d$. We show that the largest isometric spheres are all the images of $\mathcal{B}_0$ under the elements in $(\Gamma_d)_\infty$.

\begin{proposition}
An isometric sphere has the largest radius if and only if it is the image of $\mathcal{B}_0$ under an element in $(\Gamma_d)_\infty$.
\end{proposition}

\begin{proof}
Obviously, the image of $\mathcal{B}_0$ under an element in $(\Gamma_d)_\infty$ has the largest radius $\sqrt{2}$. On the contrary, given an element $G$ as the form (\ref{eq:2-6}) such that $G(q_\infty)\neq q_\infty$, then the isometric sphere of $G$ has the largest radius which leads to $g=1$. So the centre of isometric sphere of $G$ is $G^{-1}(\infty)=(\bar{h},2\Im m\bar{j},0)$ in horospherical coordinates. Since $\bar{h}$ and $2\Im m\bar{j}\in\mathcal{O}_d$, we can take a Heisenberg translation $T\in(\Gamma_d)_\infty$ mapping the origin to $(\bar{h},2\Im m\bar{j})$. Writing $T'=GTI_0$, we know that $T'$ fixes $\infty$. We conclude explicitly that the isometric sphere of $G$ is
$$\left\{\mathbf{z}\in\textbf{H}^2_{\mathbb{C}}:|\langle \mathbf{z},q_{\infty}\rangle|=|\langle \mathbf{z},G^{-1}(q_{\infty})\rangle|=|\langle \mathbf{z},TI_0(q_{\infty})\rangle|\right\},$$
which is the image of $\mathcal{B}_0$ under $T$.
\end{proof}

Our method is based on the special feature that the orbifold $\mathbf{H}^2_\mathbb{C}/\Gamma_d$ has only one cusp for $d=2,7,11$. For these types of orbifolds, one would like to construct a fundamental domain using the Ford domain (that is the exteriors of isometric spheres of all elements not fixing infinity), namely, the intersection of the Ford domain and a fundamental domain for the stabiliser of infinity. The Ford domain is canonical, but we can choose a fundamental domain for the stabiliser freely. As the first step toward the construction of a fundamental domain, we should always determine the generators of the group. In the previous section, we found suitable generators of the stabliser and constructed a fundamental domain. We will show that adjoining $I_0$ to $(\Gamma_d)_\infty$ gives the Euclidean Picard modular groups $\Gamma_d$. The basic idea of the proof can be described easily. Analogous to Theorem 3.5 of \cite{FP} we shall prove that $\langle R^{(d)}_1,R^{(d)}_2,R^{(d)}_3,T^{(d)},I_0\rangle $ has only one cusp. The fact that $PU(2,1;\mathcal{O}_d)$ has the same cusp and the stabiliser of infinity as the group generated by $R^{(d)}_1,R^{(d)}_2,R^{(d)}_3,T^{(d)},I_0$ shows that they are the same. The key step is to find a union of isometric spheres such that a fundamental domain for $\Gamma_d$ is contained in the intersection of their exteriors and a fundamental domain for the stabiliser, which implies that the group $\langle R^{(d)}_1,R^{(d)}_2,R^{(d)}_3,T^{(d)},I_0\rangle$ has only one cusp. In other words, we want to show the union of the boundaries of these isometric spheres in Heisenberg group contains each of the prisms we constructed above.  The problem of determining this will be discussed in the next section.

A simple lemma will be used in the proof of our theorems many times, we state it as follows.

\begin{lemma}(c.f. \cite{FFP})\ All Cygan balls are affinely convex.
\end{lemma}

Our aim is to prove the following results, we summarise them as three theorems.

\begin{theorem} Let $\mathcal {K}=\mathbb{Q}(\sqrt{-2})$ and let $\mathcal{O}_2=\mathbb{Z}[i\sqrt{2}]$. Then the group $PU(2,1,\mathcal{O}_2)$ is generated by the elements
$$I_0=\left[\begin{array}{ccc}0&0&1\\0&-1&0\\1&0&0\end{array}\right], R^{(2)}_1=\left[\begin{array}{ccc}1&0&0\\0&-1&0\\0&0&1\end{array}\right],
 R^{(2)}_2=\left[\begin{array}{ccc}1&2&-2\\0&-1&2\\0&0&1\end{array}\right],$$
$$ R^{(2)}_3=\left[\begin{array}{ccc}1&-i\sqrt{2}&-1\\0&-1&i\sqrt{2}\\0&0&1\end{array}\right]
\quad\text{and}\quad T^{(2)}=\left[\begin{array}{ccc}1&0&i\sqrt{2}\\0&1&0\\0&0&1\end{array}\right]
.$$
\end{theorem}

\medskip

\begin{theorem} Let $\mathcal {K}=\mathbb{Q}(\sqrt{-7})$ and let $\mathcal{O}_7=\mathbb{Z}[\omega_7]$, where $\omega_7=\frac{1}{2}(1+i\sqrt{7})$, be the ring of integers of $\mathcal{K}$. Then the group $PU(2,1,\mathcal{O}_7)$ is generated by the elements
$$I_0=\left[\begin{array}{ccc}0&0&1\\0&-1&0\\1&0&0\end{array}\right], R^{(7)}_1=\left[\begin{array}{ccc}1&0&0\\0&-1&0\\0&0&1\end{array}\right],
 R^{(7)}_2=\left[\begin{array}{ccc}1&1&-\bar{\omega}_7\\0&-1&1\\0&0&1\end{array}\right],$$
 $$
 R^{(7)}_3=\left[\begin{array}{ccc}1&\bar{\omega}_7&-1\\0&-1&\omega_7\\0&0&1\end{array}\right]\quad\text{and}\quad T^{(7)}=\left[\begin{array}{ccc}1&0&i\sqrt{7}\\0&1&0\\0&0&1\end{array}\right]
.$$
\end{theorem}

\medskip

\begin{theorem} Let $\mathcal {K}=\mathbb{Q}(\sqrt{-11})$ and let $\mathcal{O}_{11}=\mathbb{Z}[\omega_{11}]$, where $\omega_{11}=\frac{1}{2}(1+i\sqrt{11})$, be the ring of integers of $\mathcal{K}$. Then the group $PU(2,1,\mathcal{O}_{11})$ is generated by the elements
$$I_0=\left[\begin{array}{ccc}0&0&1\\0&-1&0\\1&0&0\end{array}\right], R^{(11)}_1=\left[\begin{array}{ccc}1&0&0\\0&-1&0\\0&0&1\end{array}\right],
 R^{(11)}_2=\left[\begin{array}{ccc}1&1&-\bar{\omega}_{11}\\0&-1&1\\0&0&1\end{array}\right],$$
 $$R^{(11)}_3=\left[\begin{array}{ccc}1&\bar{\omega}_{11}&-1-\bar{\omega}_{11}\\0&-1&\omega_{11}\\0&0&1\end{array}\right]\quad\text{and}\quad T^{(11)}=\left[\begin{array}{ccc}1&0&i\sqrt{11}\\0&1&0\\0&0&1\end{array}\right]
.$$
\end{theorem}

\begin{remark}For other values of $d$ such that $\mathcal{O}_d$ has class number one, namely $d=19,$ $43,$ $67,$ $163$, we can construct the same type of fundamental domain for $(\Gamma_d)_\infty$ in Heisenberg group as $(\Gamma_{11})_\infty$. Although all generators as the above types lie in $PU(2,1;\mathcal{O}_d)$, there is no reason why adjoining $I_0$ to $(\Gamma_d)_\infty$ should continue to generate the full group $PU(2,1;\mathcal{O}_d)$. From the point of view of geometric, the largest radius of isometric sphere is always $\sqrt{2}$ while the shape of the fundamental domain and the length of Heisenberg translations become large as $d$ getting large. In contrast the radius of isometric spheres other than the largest are going to be smaller and smaller. Consequently the mount of isometric spheres containing the fundamental domain increases so rapidly with the value of $d$ that it seem to be done by another way. Furthermore, the method of \cite{FFLP} could not be extended to non-Euclidean Picard modular groups.
\end{remark}

\bigskip

\section{The determination of isometric spheres}

Recall that the Cygan sphere $\mathcal{B}_0$ is the isometric sphere of $I_0$. The boundary of $\mathcal{B}_0$ is called a spinal sphere \cite{Go} in Heisenberg group, we denote by $\mathcal{S}_0$ which is defined by
\begin{equation}\label{eq:5-1}\mathcal{S}_0=\left\{(\zeta,t):\left||\zeta|^2+it\right|=2\right\}.\end{equation}

Indeed we only need to consider the boundaries of isometric spheres in Heisenberg group because two isometric spheres have a non-empty interior intersection if and only if the boundaries have a non-empty interior intersection.

\subsection{The case $\mathcal{O}_2$}

In the cases of $PU(2,1;\mathcal{O}_1)$ and $PU(2,1;\mathcal{O}_3)$, all the vertices of the fundamental domain for the stabiliser
of $q_\infty$ acting on $\partial\textbf{H}^2_\mathbb{C}$ lie inside $\mathcal{S}_0$.
For the group $PU(2,1;\mathcal{O}_2)$, it is not hard to show that six vertices of the prism $\mathbf{\Sigma}_2$ lie outside $\mathcal{S}_0$.
Therefore we need to find more isometric spheres whose boundaries together with $\mathcal{S}_0$ contain the prism $\mathbf{\Sigma}_2$.

We consider the map

\begin{equation*}
I_0R^{(2)}_2I_0=\left[\begin{array}{ccc}
1&0&0\\-2&-1&0\\-2&-2&1
\end{array}\right],
\end{equation*}
whose isometric sphere which we denote by $\mathcal{B}_1$ is a Cygan sphere centred at the point $(1,0,0)$ (in horospherical coordinates) with radius 1. The boundary of $\mathcal{B}_1$ is given by
\begin{equation}\label{eq:5-2}\mathcal{S}_1=\left\{(\zeta,t):\left||\zeta-1|^2+it+2i\Im m\zeta\right|=1\right\}.\end{equation}

Minimising the number of spinal spheres by the symmetry of $R^{(2)}_1$, it suffice to consider $\mathcal{S}_0$ and several images of $\mathcal{S}_1$ under some suitable elements in $(\Gamma_2)_\infty$, these are in Heisenberg coordinates given by
\begin{align*}
T^{(2)}(\mathcal{S}_1)&=\left\{(\zeta,t):\left||\zeta-1|^2+it-2i\sqrt{2}+2i\Im m\zeta\right|=1\right\},\\
{T^{(2)}}^{-1}(\mathcal{S}_1)&=\left\{(\zeta,t):\left||\zeta-1|^2+it+2i\sqrt{2}+2i\Im m\zeta\right|=1\right\},\\
R^{(2)}_1(\mathcal{S}_1)&=\left\{(\zeta,t):\left||\zeta+1|^2+it-2i\Im m\zeta\right|=1\right\},\\
{T^{(2)}}^{-1}R^{(2)}_1(\mathcal{S}_1)&=\left\{(\zeta,t):\left||\zeta+1|^2+it+2i\sqrt{2}-2i\Im m\zeta\right|=1\right\}.
\end{align*}
We claim that the prism $\mathbf{\Sigma}_2$ lies inside the union of $\mathcal{S}_0$ and these images of $\mathcal{S}_1$, see Figure 5.1 for viewing these spinal spheres.

\begin{figure}
\centering
  \subfigure[]{
    \label{fig:mini:subfig:1}
 \begin{minipage}[1]{0.42\textwidth} \centering
      \includegraphics[width=2in]{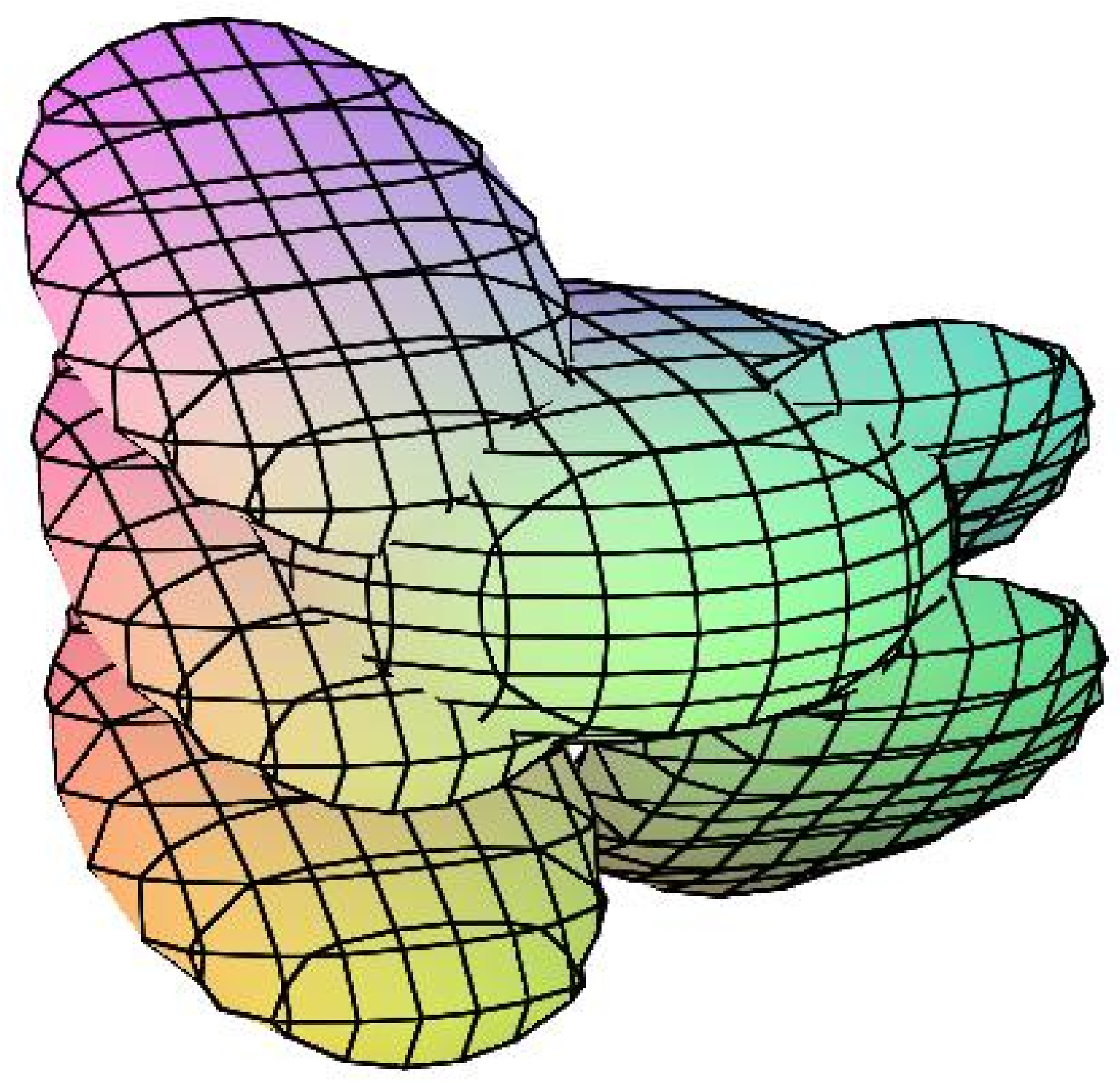}
  \end{minipage}
}%
  \subfigure[]{
    \label{fig:mini:subfig:2} 
    \begin{minipage}[3]{0.42\textwidth}
      \centering
      \includegraphics[width=2in]{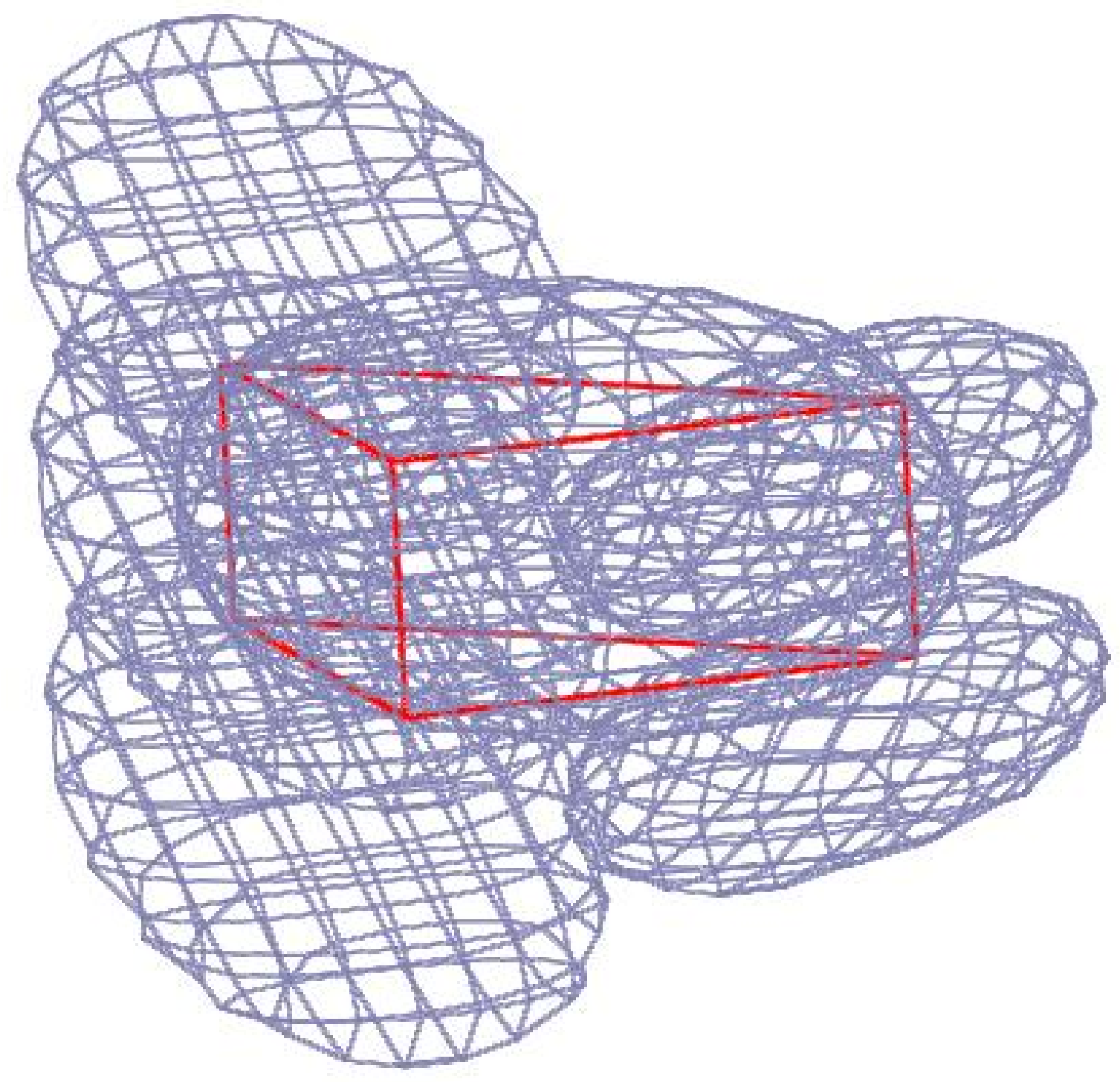}
    \end{minipage}
}
\medskip
\makeatletter\def\@captype{figure}\makeatother
\bigskip
\caption{(a)\ The shading view of neighboring spinal spheres containing the fundamental domain for $(\Gamma_2)_\infty$.
(b)\ Another view for these spinal spheres.}
  \label{fig:mini:subfig} 
\end{figure}

\begin{proposition}The prism $\mathbf{\Sigma}_2$ is contained in the union of the interiors of the spinal spheres $\mathcal{S}_0,$ $\mathcal{S}_1,$ $T^{(2)}(\mathcal{S}_1),$ ${T^{(2)}}^{-1}(\mathcal{S}_1),$ $R^{(2)}_1(\mathcal{S}_1)$
and ${T^{(2)}}^{-1}R^{(2)}_1(\mathcal{S}_1)$.
\end{proposition}

\begin{proof} It suffices to show there exists three points $(v_1^+)^{(j)}$ $(j=1,2,3)$ on the edges $(v_1^+,v_1^-),$ $(v_1^+,v_2^+)$ and $(v_1^+,v_3^+)$
which lie in the intersection of the interiors of $\mathcal{S}_0$ and $\mathcal{S}_1$ such that the tetrahedron $\mathbb{T}(v_1^+)$ with vertices
 $v_1^+,$ $(v_1^+)^{(1)},$ $(v_1^+)^{(2)},$ $(v_1^+)^{(3)}$ lies inside $\mathcal{S}_1$. By the same argument,
we can also obtain other five tetrahedra $\mathbb{T}(v_2^+),$ $\mathbb {T}(v_3^+),$ $\mathbb{T}(v_1^-),$ $\mathbb {T}(v_2^-),$ $\mathbb{T}(v_3^-)$
with apex $v_2^+,$ $v_3^+,$ $v_1^-,$ $v_2^-,$ $v_3^-$ respectively such that
 $\mathbb {T}(v_2^+)\in Int(T^{(2)}(\mathcal{S}_1)),$ $\mathbb {T}(v_3^+)\in Int(R^{(2)}_1(\mathcal{S}_1)),$ $\mathbb {T}(v_1^-)\in Int({T^{(2)}}^{-1}(\mathcal{S}_1)),$
  $\mathbb {T}(v_2^-)\in Int(\mathcal{S}_1)$ and $\mathbb {T}(v_3^-)\in Int({T^{(2)}}^{-1}R^{(2)}_1(\mathcal{S}_1))$. Moreover, the core part obtained by cutting off six the tetrahedra from the prism lies inside $\mathcal{S}_0$.

We shall prove the existence of the tetrahedron $\mathbb {T}(v_1^+)$ and the others follow similarly. The edge joining $v_1^+$ and $v_1^-$ is contained the complex line $\zeta=1-\sqrt{2}i/2$ which is given by points with Heisenberg coordinates
$$\zeta=1-\sqrt{2}i/2,\quad-\sqrt{2}\leq t\leq\sqrt{2}.$$
The edge joining $v_1^+$ and $v_2^+$ is given by points with Heisenberg coordinates
$$\Re e\zeta=1,\quad-\sqrt{2}/2\leq\Im m\zeta\leq\sqrt{2}/2,\quad t=\sqrt{2}.$$
The edge joining $v_1^+$ and $v_3^+$ is given by points with Heisenberg coordinates
$$\Re e\zeta=-\sqrt{2}\Im m\zeta,\quad t=\sqrt{2}.$$
From (\ref{eq:5-1}) and (\ref{eq:5-2}), the points on the edge $(v^+_1,v^-_1)$ lie in the intersection of the interiors of $\mathcal{S}_0$ and $\mathcal{S}_1$ if and only if
\begin{equation}\label{eq:5-3}
\left|3/2+it\right|<2\quad\text{and}\quad\left|1/2-(t-\sqrt{2})i\right|<1.
\end{equation}
By an easy calculation, the inequalities (\ref{eq:5-3}) are equivalent to $$\sqrt{2}-\sqrt{3}/2<t<\sqrt{7}/2.$$

Using the same argument as above, we obtain that the points on the edge $(v^+_1,v^+_2)$ lie in the intersection of the interiors of $\mathcal{S}_0$ and $\mathcal{S}_1$ if and only if $\Re e\zeta=1,$ and $-\sqrt{\sqrt{2}-1}<\Im m\zeta<\delta_1$, where $\delta_1\approx-0.208$ is the largest real root of the equation $x^4+4x^2+4\sqrt{2}x+1=0$. The points on the edge $(v^+_1,v^+_3)$ lie in the intersection of the interiors of $\mathcal{S}_0$ and $\mathcal{S}_1$ if and only if $\Re e\zeta=-\sqrt{2}\Im m\zeta$ and $-2^{1/4}/\sqrt{3}<\Im m\zeta<\delta_2$, where $\delta_2\approx-0.264$ is the largest real root of the equation $9x^4+12\sqrt{2}x^3+18x^2+8\sqrt{2}x+2=0$.

In term of these, we choose three points as $(v_1^+)^{(1)}=(1-\sqrt{2}i/2,1)$ on the edge $(v^+_1,v^-_1)$, $(v_1^+)^{(2)}=(1-i/2,\sqrt{2})$ on the edge $(v^+_1,v^+_2)$ and $(v_1^+)^{(3)}=(\sqrt{2}/2-i/2,\sqrt{2})$ on the edge $(v^+_1,v^+_3)$, which are inside the intersection of the interiors of $\mathcal{S}_0$, $\mathcal{S}_1$ and also the vertex $v_1^+$ lies inside $\mathcal{S}_1$. Therefore the tetrahedron $\mathbb {T}(v_1^+)$ with the vertices $v_1^+,$ $(v_1^+)^{(1)},$ $(v_1^+)^{(2)},$ $(v_1^+)^{(3)}$ lies inside $\mathcal{S}_1$ by Lemma 4.1.
\end{proof}

\subsection{The case $\mathcal{O}_7$}

In this case, the distance between the top and base of the fundamental domain for the
stabiliser $(\Gamma_{7})_\infty$ is greater than the diameter of $\mathcal{S}_0$, which implies that the prism $\mathbf{\Sigma}_7$ can not be
contained inside $\mathcal{S}_0$ completely. Due to increasing the length of Heisenberg translations, only the images of $\mathcal{S}_0$ under the elements in $(\Gamma_7)_\infty$ could not cover the whole prism. We show that there are also isometric spheres with Cygan radius smaller than $\sqrt{2}$ whose centres are near to origin.

Therefore we consider the map
$$Q = I_0R^{(7)}_2I_0=\left[\begin{array}{ccc}
1&0&0\\
1&1&0\\
\bar{\omega}_7&1&1
\end{array}\right].
$$
Consider the isometric spheres of $Q$ and $Q^{-1}$, which we denote by $\mathcal{B}_2$ and $\mathcal{B}_3$, respectively.
The centre of $\mathcal{B}_2$ is $Q^{-1}(\infty)$, which is the point with horospherical coordinates
$(1/4+i\sqrt{7}/4,\sqrt{7}/2,0)$ and the centre of $\mathcal{B}_3$, is $Q(\infty)$ which has horospherical coordinates $(1/4-i\sqrt{7}/4,\sqrt{7}/2,0)$. Both these isometric spheres have Cygan radius $\sqrt{2/|\omega_7|}=2^{1/4}$. The boundaries of these isometric spheres $\mathcal{B}_2$ and $\mathcal{B}_3$ are in Heisenberg coordinates given by
\begin{align}\label{eq:5-4}
\mathcal{S}_2&=\left\{(\zeta,t):\left|\left|\zeta-\omega_7/2\right|^2+it+i\sqrt{7}/2+i\Im m(\bar{\omega}_7\zeta)\right|=\sqrt{2}\right\},\\ \label{eq:5-5}
\mathcal{S}_3&=\left\{(\zeta,t):\left|\left|\zeta-\bar{\omega}_7/2\right|^2+it-i\sqrt{7}/2+i\Im m(\omega_7\zeta)\right|=\sqrt{2}\right\}.
\end{align}

In order to cover the prim $\mathbf{\Sigma}_7$ by the spinal spheres, we use the symmetry property of $R^{(7)}_1$, it suffice to consider $\mathcal{S}_0,\mathcal{S}_2$ and images of $\mathcal{S}_0$ and $\mathcal{S}_3$ under suitable elements in $(\Gamma_7)_\infty$, these are points with Heisenberg coordinates given by
\begin{align*}
R^{(7)}_2(\mathcal{S}_0)&=\left\{(\zeta,t):\left||\zeta-1|^2+it-i\sqrt{7}+2i\Im m\zeta\right|=4\right\},\\
{R^{(7)}_2}^{-1}(\mathcal{S}_0)&=\left\{(\zeta,t):\left||\zeta-1|^2+it+i\sqrt{7}+2i\Im m\zeta\right|=4\right\},\\
{R^{(7)}_2}^{-1}(\mathcal{S}_3)&=\left\{(\zeta,t):\left|\left|\zeta-(1+\omega_7)/2\right|^2+it+i\sqrt{7}\right.\right.\\
&\hspace{5cm}\left.\left.+i\Im m((1+\bar{\omega}_7)\zeta)\right|=\sqrt{2}\right\},\\
R^{(7)}_3R^{(7)}_2(\mathcal{S}_3)&=\left\{(\zeta,t):\left|\left|\zeta+\bar{\omega}_7/2\right|^2+it-i\sqrt{7}/2-i\Im m(\omega_7\zeta)\right|=\sqrt{2}\right\},\\
R^{(7)}_1R^{(7)}_3R^{(7)}_2(\mathcal{S}_3)&=\left\{(\zeta,t):\left|\left|\zeta-\bar{\omega}_7/2\right|^2+it-i\sqrt{7}/2+i\Im m(\omega_7\zeta)\right|=\sqrt{2}\right\}.
\end{align*}

We claim that the prism $\mathbf{\Sigma}_7$ lies inside the union of $\mathcal{S}_0,$ $\mathcal{S}_2$ and these
 images $R^{(7)}_2(\mathcal{S}_0),$ ${R^{(7)}_2}^{-1}(\mathcal{S}_0),$ ${R^{(7)}_2}^{-1}(\mathcal{S}_3)$, $R^{(7)}_3R^{(7)}_2(\mathcal{S}_3),$ $R^{(7)}_1R^{(7)}_3R^{(7)}_2(\mathcal{S}_3)$, see Figure 5.2 for viewing these spinal spheres.

\begin{figure}
\centering
  \subfigure[]{
    \label{fig:mini:subfig:1}
 \begin{minipage}[1]{0.42\textwidth} \centering
      \includegraphics[width=2in]{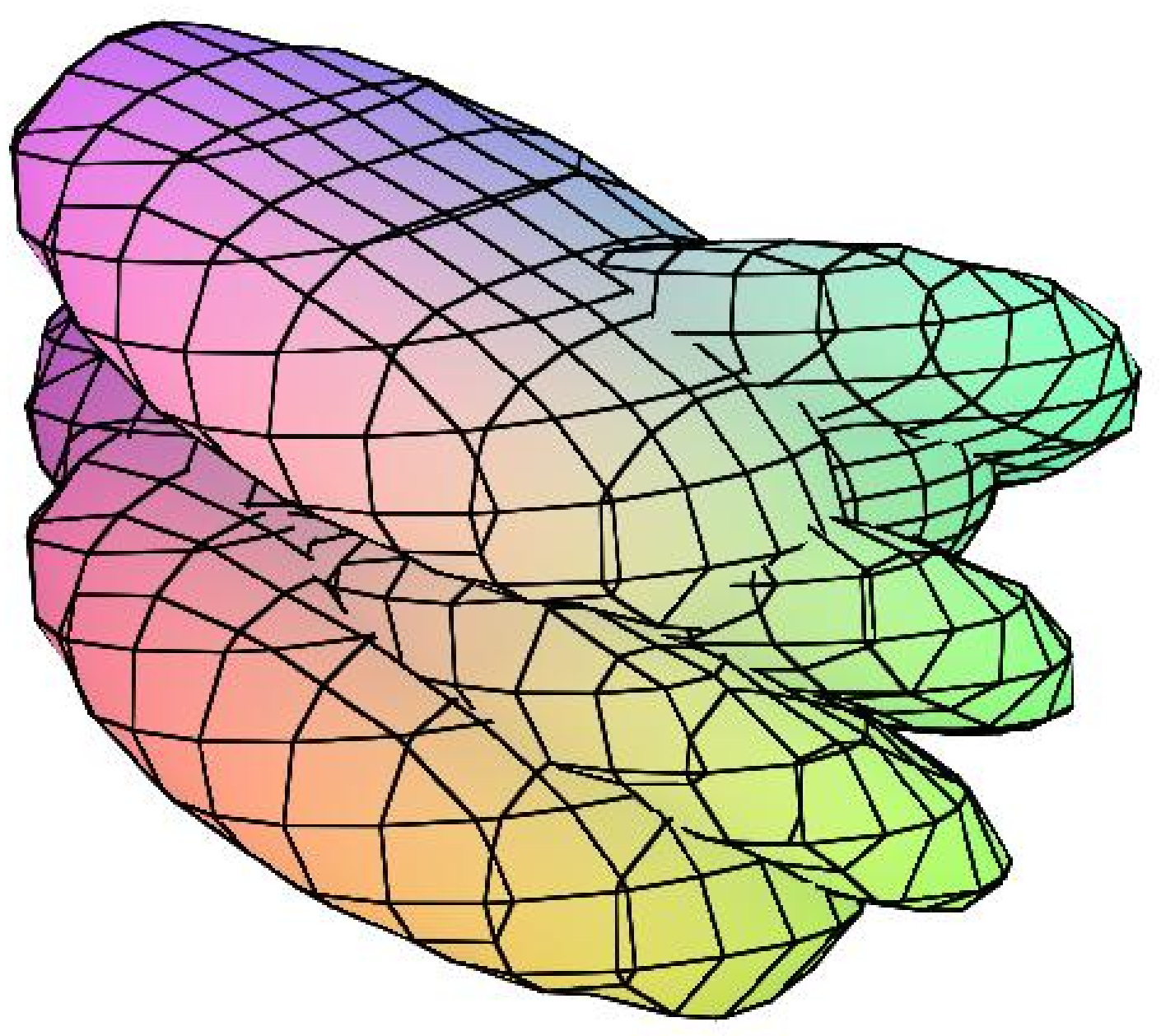}
  \end{minipage}
}%
  \subfigure[]{
    \label{fig:mini:subfig:2} 
    \begin{minipage}[3]{0.42\textwidth}
      \centering
      \includegraphics[width=2in]{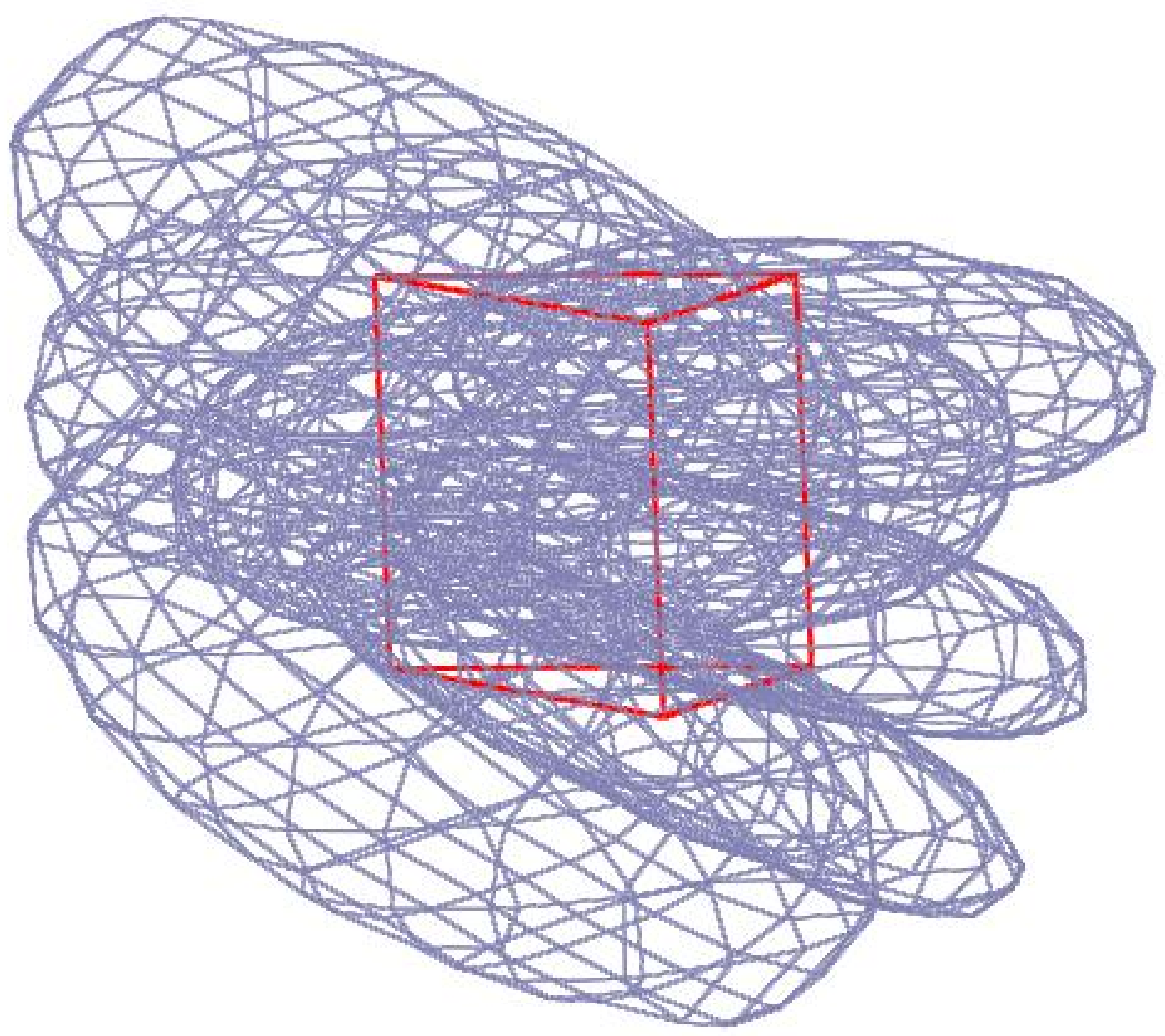}
    \end{minipage}
}
\medskip
\makeatletter\def\@captype{figure}\makeatother
\bigskip
\caption{(a)\ The shading view of neighboring spinal spheres containing the fundamental domain for $(\Gamma_7)_\infty$. (b)\ Another view for these spinal spheres.}
  \label{fig:mini:subfig} 
\end{figure}

\begin{proposition}\label{eq:5.2}The prism $\mathbf{\Sigma}_7$ is contained in the union of the interiors of the spinal spheres
 $\mathcal{S}_0,$ $\mathcal{S}_2,$ $R^{(7)}_2(\mathcal{S}_0),$ ${R^{(7)}_2}^{-1}(\mathcal{S}_0),$ ${R^{(7)}_2}^{-1}(\mathcal{S}_3),$ $R^{(7)}_3R^{(7)}_2(\mathcal{S}_3)$  and $R^{(7)}_1R^{(7)}_3R^{(7)}_2(\mathcal{S}_3)$.
\end{proposition}

\begin{proof}It suffice to show that the prism $\mathbf{\Sigma}_7$ can be decomposed into several pieces as polyhedra such that each polyhedron lies inside a spinal sphere which is described in the proposition and the common face of two adjacent polyhedra lie in the intersection of the interior of two spinal spheres which contain these two polyhedra.

We need to add sixteen points on the faces of the prism $\mathbf{\Sigma}_7$ in order to decompose the prim into seven polyhedra, in Heisenberg coordinates, these are given by
$$
\begin{array}{llll}
p_1=(1/4-i\sqrt{7}/4,3/2),&p_2=(0.11-i11\sqrt{7}/100,1.44+\sqrt{7}/50),\\
p_3=(1/2,8/5),&p_4=(-1/10+i\sqrt{7}/10,\sqrt{7}),\\
p_5=(-1/10+i\sqrt{7}/10,\sqrt{7}/2),&p_6=(3/4+i\sqrt{7}/4,1.7),\\
p_7=(-1/4+i\sqrt{7}/4,1),&p_8=(1/4-i\sqrt{7}/4,-1),\\
p_9=(1/60-i\sqrt{7}/60,-2\sqrt{7}/3),&p_{10}=(-1/20+i\sqrt{7}/20,-\sqrt{7}),\\
p_{11}=(3/5+i\sqrt{7}/10,-2\sqrt{7}/3),&p_{12}=(7/10+i\sqrt{7}/5,-\sqrt{7}),\\
p_{13}=(3/4+i\sqrt{7}/4,-2\sqrt{7}/3),&p_{14}=(5/12+i\sqrt{7}/4,-2\sqrt{7}/3),\\
p_{15}=(1/4+i\sqrt{7}/4,-\sqrt{7}),&p_{16}=(-1/4+i\sqrt{7}/4,-1).
\end{array}$$

\begin{figure}
\centering
\psfrag{T}{\Large$\mathbb{T}$}\psfrag{P1}{\Large$\mathbb{P}_1$}\psfrag{P2}{\Large$\mathbb{P}_2$}
\psfrag{P3}{\Large$\mathbb{P}_3$}\psfrag{H1}{\Large$\mathbb{H}_1$}\psfrag{H2}{\Large$\mathbb{H}_2$}
\psfrag{O}{\Large$\mathbb{O}$}
\psfrag{v1+}{\Large$v_1^+$}\psfrag{v2+}{\Large$v_2^+$}\psfrag{v4+}{\Large$v_4^+$}
\psfrag{v1-}{\Large$v_1^-$}\psfrag{v2-}{\Large$v_2^-$}\psfrag{v4-}{\Large$v_4^-$}
\psfrag{p1}{\Large$p_1$}\psfrag{p2}{\Large$p_2$}\psfrag{p3}{\Large$p_3$}\psfrag{p4}{\Large$p_4$}
\psfrag{p5}{\Large$p_5$}\psfrag{p6}{\Large$p_6$}\psfrag{p7}{\Large$p_7$}\psfrag{p8}{\Large$p_8$}
\psfrag{p9}{\Large$p_9$}\psfrag{p10}{\Large$p_{10}$}\psfrag{p11}{\Large$p_{11}$}\psfrag{p12}{\Large$p_{12}$}
\psfrag{p13}{\Large$p_{13}$}
\psfrag{p14}{\Large$p_{14}$}
\psfrag{p15}{\Large$p_{15}$}
\psfrag{p16}{\Large$p_{16}$}
\resizebox{2in}{!}{\includegraphics{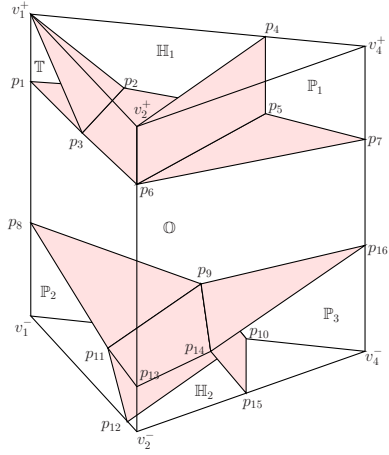}}
\medskip
\makeatletter\def\@captype{figure}\makeatother
\caption{A view of the decomposition for the prism $\mathbf{\Sigma}_7$ as several polyhedra.}
\end{figure}

We describe these polyhedra as follows:

(i)\ The tetrahedron $\mathbb{T}$ with the vertice $v^+_1,$ $p_1,$ $p_2,$ $p_3$;

(ii)\ The hexahedron $\mathbb{H}_1$ with the vertice $v^+_1,$ $v_2^+,$ $p_2,$ $p_3,$ $p_4,$ $p_5,$ $p_6$;

(iii)\ The pentahedron $\mathbb{P}_1$ with the vertice $v_2^+,$ $p_4,$ $p_5,$ $p_6,$ $v^+_4,$ $p_7$;

(iv)\ The pentahedron $\mathbb{P}_2$ with the vertice  $v_1^-,$ $p_8,$ $p_9,$ $p_{10},$ $p_{11},$ $p_{12}$;

(v)\ The hexahedron $\mathbb{H}_2$ with the vertice $p_9,$ $p_{10},$ $p_{11},$ $p_{12},$ $p_{13},$ $v_2^-,$ $p_{14},$ $p_{15}$;

(vi)\ The pentahedron $\mathbb{P}_3$ with the vertice $p_9,$ $p_{10},$ $p_{14},$ $p_{15},$ $p_{16},$ $v_4^+$;

(vii)\ The octahedron $\mathbb{O}$ with the vertice $p_1,$ $p_2,$ $p_3,$ $p_5,$ $p_6,$ $p_7,$ $p_8,$ $p_9,$ $p_{11},$ $p_{13},$ $p_{14},$ $p_{16}$.
                                                      
Note that the face $(p_1,p_2,p_3)$ of $\mathbb{T}$ and the face $(p_2,p_3,p_5,p_6)$ of $\mathbb{H}_1$ are on the face $(p_1,p_5,p_6)$ of $\mathbb{O}$; the common face $(v^+_2,p_4,p_5,p_6)$ of $\mathbb{H}_1$ and $\mathbb{P}_1$ is a vertical plane; the face $(p_9,p_{11},p_{13},p_{14})$ of $\mathbb{H}_2$ is parallel to the base of the prism. Furthermore, the faces $(p_9,p_{10},p_{11},p_{12})$ and $(p_9,p_{10},p_{14},p_{15})$ are the trapeziums since 
the edge $(p_9,p_{11})$ is parallel to $(p_{10},p_{12})$ and the edge $(p_9,p_{14})$ is parallel to $(p_{10},p_{15})$.

By examining the location of the points and applying Lemma 4.1, we conclude that the tetrahedron $\mathbb{T}$ is inside the spinal sphere
 $R^{(7)}_1R^{(7)}_3R^{(7)}_2(\mathcal{S}_3)$; the hexahedron $\mathbb{H}_1$ is contained inside the spinal sphere $R^{(7)}_2(\mathcal{S}_0)$; the pentahedron $\mathbb{P}_1$ is inside  $R^{(7)}_3R^{(7)}_2(\mathcal{S}_3)$; the pentahedron $\mathbb{P}_2$ is contained inside
 ${R^{(7)}_2}^{-1}(\mathcal{S}_0)$; the hexahedron $\mathbb{H}_2$ is contained inside ${R^{(7)}_2}^{-1}(\mathcal{S}_3)$;
  the pentahedron $\mathbb{P}_3$ is inside $\mathcal{S}_2$;  the remaining octahedron $\mathbb{O}$ is inside $\mathcal{S}_0$; see Figure 5.3 for viewing the decomposition of the prism.
\end{proof}

\subsection{The case $\mathcal{O}_{11}$}

In this case, we know that the fundamental domain for the stabiliser
$(\Gamma_{11})_\infty$ cannot be still inside $\mathcal{S}_0$ completely. The radius of spinal spheres other than the largest are so small that these spinal spheres are not much contribution to covering the prism $\mathbf{\Sigma}_{11}$. Due to the different shape of the prism $\mathbf{\Sigma}_{11}$ with the case $\mathcal{O}_7$, we only need to consider the largest spinal spheres which are the images of $\mathcal{S}_0$ under the elements of $(\Gamma_{11})_\infty$. In order to determine a union of the spinal spheres which covering the prim $\mathbf{\Sigma}_{11}$,  we minimise their numbers by the symmetry of $R^{(11)}_1$, it suffice to consider $\mathcal{S}_0$ and the images of $\mathcal{S}_0$ under suitable elements in $(\Gamma_{11})_\infty$, these are in Heisenberg coordinates given by
\begin{eqnarray*}
T^{(11)}(\mathcal{S}_0)&=&\left\{(\zeta,t):\left||\zeta|^2+it-2i\sqrt{11}+2i\Im m\zeta\right|=4\right\},\\
R^{(11)}_2(\mathcal{S}_0)&=&\left\{(\zeta,t):\left||\zeta-1|^2+it-i\sqrt{11}+2i\Im m\zeta\right|=4\right\},\\
R^{(11)}_1R^{(11)}_2(\mathcal{S}_0)&=&\left\{(\zeta,t):\left||\zeta+1|^2+it-i\sqrt{11}-2i\Im m\zeta\right|=4\right\},\\
{R^{(11)}_2}^{-1}(\mathcal{S}_0)&=&\left\{(\zeta,t):\left||\zeta-1|^2+it+i\sqrt{11}+2i\Im m\zeta\right|=4\right\},\\
R^{(11)}_3(\mathcal{S}_0)&=&\left\{(\zeta,t):\left|\left|\zeta-\omega_{11}\right|^2+it-i\sqrt{11}-2i\Im m(\bar{\omega}_{11}\zeta)\right|=4\right\},\\
{R^{(11)}_3}^{-1}(\mathcal{S}_0)&=&\left\{(\zeta,t):\left|\left|\zeta-\omega_{11}\right|^2+it+i\sqrt{11}-2i\Im m(\bar{\omega}_{11}\zeta)\right|=4\right\},\\
R^{(11)}_3R^{(11)}_2(\mathcal{S}_0)&=&\left\{(\zeta,t):\left|\left|\zeta+\bar{\omega}_{11}\right|^2+it-i\sqrt{11}-2i\Im m(\omega_{11}\zeta)\right|=4\right\},\\
R^{(11)}_1R^{(11)}_3R^{(11)}_2(\mathcal{S}_0)&=&\left\{(\zeta,t):\left||\zeta-\bar{\omega}_{11}|^2+it-i\sqrt{11}+2i\Im m(\omega_{11}\zeta)\right|=4\right\},\\
R^{(11)}_1{R^{(11)}_3}^{-1}R^{(11)}_2(\mathcal{S}_0)&=&\left\{(\zeta,t):\left||\zeta-\bar{\omega}_{11}|^2+it+i\sqrt{11}+2i\Im m(\omega_{11}\zeta)\right|=4\right\}.
\end{eqnarray*}

\begin{definition}
Let $X$ be a closed polygonal chain (not necessarily in a plane), then a topological disk defined by the cone over $X$ with apex $v$ is called a \textbf{cone-polygon}, denoted by $\mathbb{D}_v(X)$.
\end{definition}

Note that a polygon in traditional sense can be interpreted as a cone-polygon, in that case, the boundary of cone-polygon and the apex lie in the same plane and moreover the apex is in the interior of the boundary. We claim that the prism $\mathbf{\Sigma}_{11}$ lies inside the union of $\mathcal{S}_0$ and its images as above, see Figure 5.4 for viewing these spinal spheres.

\begin{figure}
\centering
  \subfigure[]{
    \label{fig:mini:subfig:1}
 \begin{minipage}[1]{0.4\textwidth} \centering
      \includegraphics[width=2in]{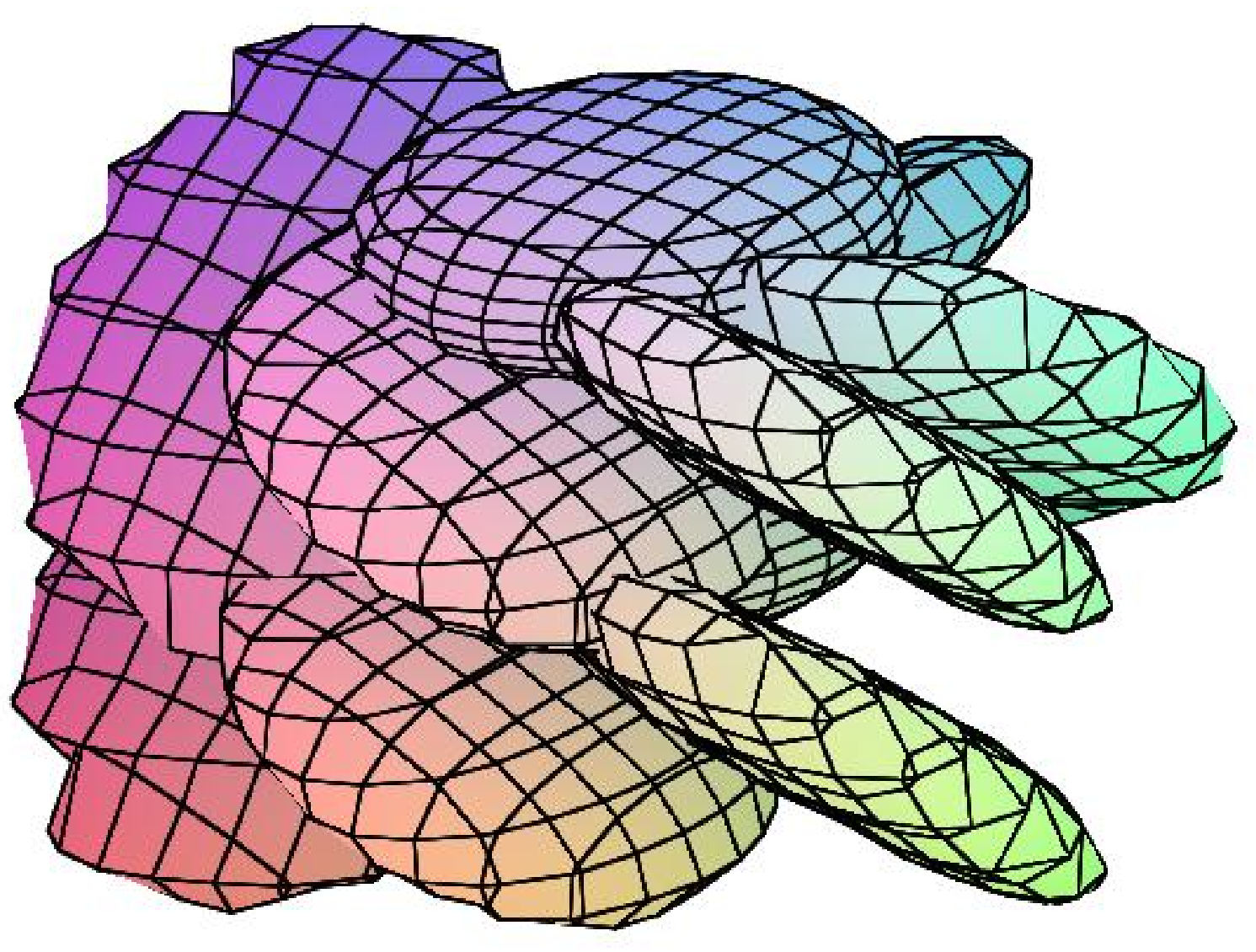}
  \end{minipage}
}%
  \subfigure[]{
    \label{fig:mini:subfig:2} 
    \begin{minipage}[3]{0.4\textwidth}
      \centering
      \includegraphics[width=2in]{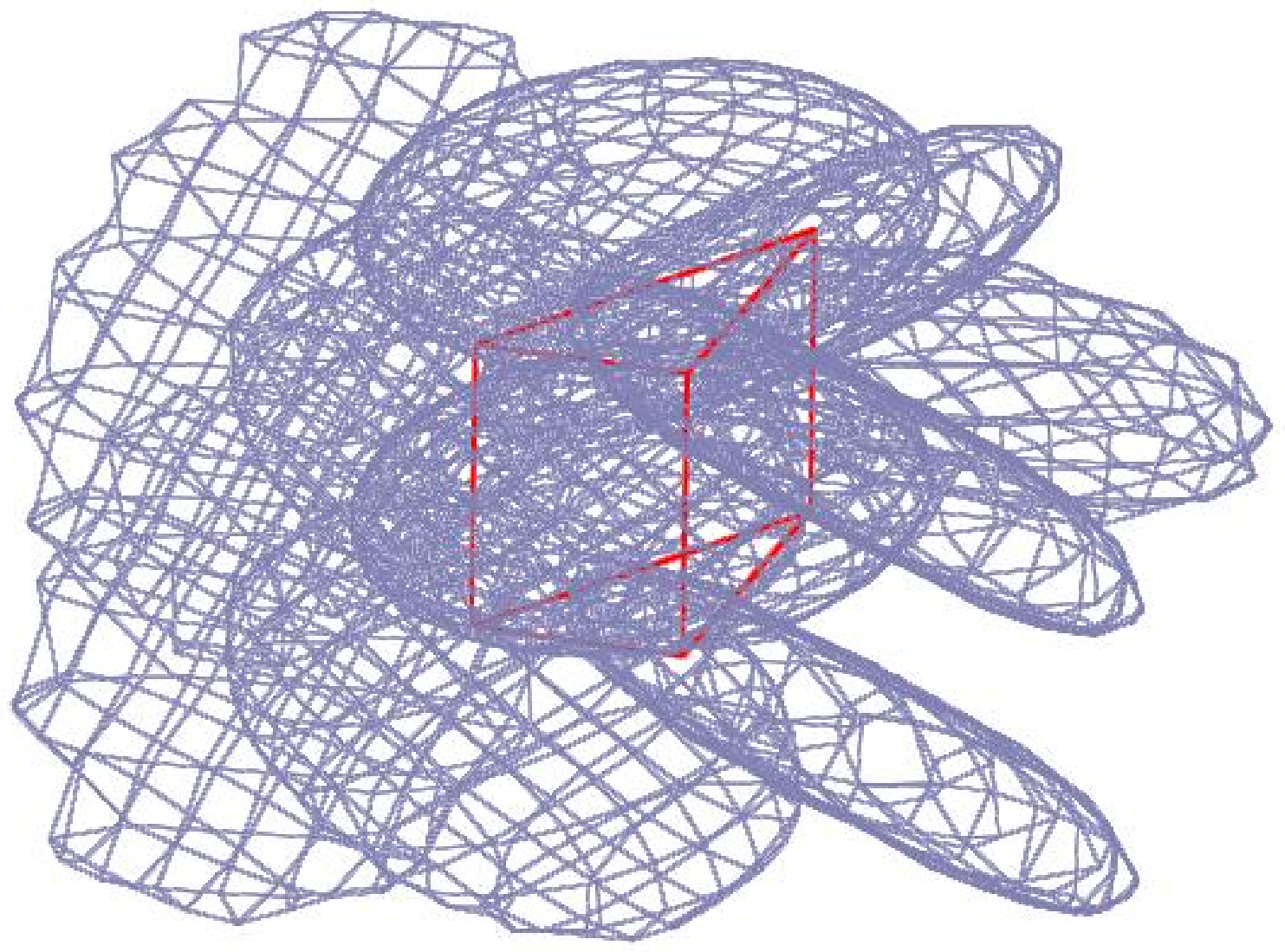}
    \end{minipage}
}
\medskip
\makeatletter\def\@captype{figure}\makeatother
\bigskip
\caption{(a)\ The shading view of neighboring spinal spheres containing the fundamental domain for $(\Gamma_{11})_\infty$. (b)\ Another view for these spinal spheres.}
  \label{fig:mini:subfig} 
\end{figure}

\begin{proposition}The prism $\mathbf{\Sigma}_{11}$ is contained in the union of the interiors of the spinal spheres $\mathcal{S}_0,$ $T^{(11)}(\mathcal{S}_0),$  $R^{(11)}_2(\mathcal{S}_0),$ ${R^{(11)}_2}^{-1}(\mathcal{S}_0),$ $R^{(11)}_3(\mathcal{S}_0),$ ${R^{(11)}_3}^{-1}(\mathcal{S}_0),$  $R^{(11)}_1{R^{(11)}_3}^{-1}R^{(11)}_2(\mathcal{S}_0)$, $R^{(11)}_1R^{(11)}_3R^{(11)}_2(\mathcal{S}_0)$, $R^{(11)}_1R^{(11)}_2(\mathcal{S}_0)$ and $R^{(11)}_3R^{(11)}_2(\mathcal{S}_0)$.
\end{proposition}

\begin{proof}
Using the same argument of Proposition \ref{eq:5.2}, we want to decompose the prism $\mathbf{\Sigma}_{11}$ into several
polyhedral cells. The difference is the complicated intersection of the spinal spheres, which leads that the prism is difficultly decomposed into several polyhedral cells each of which is contained in one spinal sphere. Observe that a union of interiors of several spinal spheres is
a star-convex set if they have a non-empty interior intersection. We shall show that the collection of these spinal spheres can be separated into several parts such that each part contains certain polyhedral cell. All these polyhedral cells are defined by the star-disk as its boundary.

We first define a tetrahedron $\mathbb{T}$ with vertices $v_1^-,q_1,q_2,q_3$, where
$$\begin{array}{l}
q_1=\left(1/4-i\sqrt{11}/4,-2\sqrt{11}/3\right),\\
q_2=\left(3/20-3i\sqrt{11}/20,-4\sqrt{11}/5\right),\\
q_3=\left(7/20-3i\sqrt{11}/20,-9\sqrt{11}/10\right).
\end{array}$$
Observe that the points $q_1,q_2,q_3$ lie on the edges $(v^+_1,v^-_1),$ $(v^-_1,v_3^-)$ and $(v_1^-,v_2^-)$, respectively.
An easy calculation shows that this tetrahedron is contained inside $R^{(11)}_1{R^{(11)}_3}^{-1}R^{(11)}_2(\mathcal{S}_0)$ by Lemma 4.1.

Next, we define a hexahedron $\mathbb{H}_1$ with vertices $q_1,q_2,q_3,q_4,q_5,q_6,q_7,v^+_0$ and another hexahedron $\mathbb{H}_2$ with vertices $v_2^-,q_5,q_6,q_7,q_8,q_9$, where
$$
\begin{array}{llll}
q_4=\left(1/4-i\sqrt{11}/4,-1/2\right),&q_5=\left(0.42+0.26i,-0.71\sqrt{11}+0.39\right),\\
q_6=\left(0.6+i\sqrt{11}/10,-0.65\sqrt{11}\right),&q_7=\left(0.58+2i\sqrt{11}/25,-1.92\right),\\
q_8=\left(3/4+i\sqrt{11}/4,0\right),&q_9=(0.55+i\sqrt{11}/4,-2\sqrt{11}/5).
\end{array}$$

Observe that the points $q_4,q_6,q_8,q_9$ lie on the edges $(v^+_1,v^-_1),$ $(v^-_1,v_2^-),$ $(v_2^+,v_2^-)$ and $(v^-_2,v^-_3)$, respectively.
The points $q_5$ lies on the interior of the base of the prism and $q_7$ lies on the interior of the face $(v^+_1,v^-_1,v^-_2,v^+_2)$. Then we know the hexahedron $\mathbb{H}_1$ has the faces $(q_1,q_2,q_3),$ $(q_1,q_3,q_6,q_7,q_4),$ $(q_1,q_2,v^+_0,q_4),$ $(q_4,q_5,v^-_0),$ $(q_4,q_5,q_7)$ and $(q_5,q_6,q_7)$ and the hexahedron $\mathbb{H}_2$ has the faces $(q_5,q_6,q_7),$ $(q_5,q_7,q_8),$ $(v^-_2,q_8,q_9),$ $(q_5,q_8,q_9)$ $(q_6,q_7,q_8,v^-_2)$ and $(q_5,q_6,v^-_2,q_9)$. By examining the location of these points and Lemma 4.1, we conclude that the hexahedron $\mathbb{H}_1$ is contained inside ${R^{(11)}_2}^{-1}(\mathcal{S}_0)$ and the hexahedron $\mathbb{H}_2$ is lied inside ${R^{(11)}_3}^{-1}(\mathcal{S}_0)$.

We focus on describing other polyhedral cells in the decomposition of the prism $\mathbf{\Sigma}_{11}$. Let $\mathcal{U}_1$ denote the union of $R^{(11)}_2(\mathcal{S}_0),$ $ R^{(11)}_1R^{(11)}_2(\mathcal{S}_0)$ and $R^{(11)}_1R^{(11)}_3R^{(11)}_2(\mathcal{S}_0)$. We verify that $q_{10}=(0.2-0.4i,2.4)$ is in the intersection of the interiors of these three spinal spheres, which implies that $\mathcal{U}_1$ is a star-convex set about $q_{11}$. Analogously, we know $\mathcal{U}_2$, denoted by the union of $T^{(11)}(\mathcal{S}_0),$ $R^{(11)}_3(\mathcal{S}_0),$ $R^{(11)}_1R^{(11)}_2(\mathcal{S}_0)$ and $R^{(11)}_3R^{(11)}_2(\mathcal{S}_0)$, is a star-convex set about $q_{11}=(0.18+0.72i,4.8)$. This can be verified by examining the location of $q_{12}$ which is in the intersection of the interiors of these four spinal spheres.
We need to add the following points on the faces of the prism $\mathbf{\Sigma}_{11}$, each of which is in the intersection of the interiors of at least two spinal spheres.
$$ \begin{array}{ll}
q_{12}=(1/4-i\sqrt{11}/4,\sqrt{11}/2),&q_{13}=(0.21-0.21i\sqrt{11},\sqrt{11}/2),\\
q_{14}=(0,\sqrt{11}/2),&q_{15}=(-0.21+0.21i\sqrt{11},\sqrt{11}/2),\\
q_{16}=(i\sqrt{11}/4,1),&q_{17}=(3/4+i\sqrt{11}/4,1),\\
q_{18}=(0.42-2i\sqrt{11}/25,1.95),&q_{19}=(3/4+i\sqrt{11}/4,\sqrt{11}),\\
q_{20}=(0.6+i\sqrt{11}/10,27\sqrt{11}/20),&q_{21}=(0.42+0.26i,1.29\sqrt{11}+0.39),\\
q_{22}=(-1.4+1.4i\sqrt{11},4\sqrt{11}/5),&q_{23}=(-1/4+i\sqrt{11}/4,\sqrt{11}/2).
\end{array}
$$

Observe that the points $q_{12},q_{20},q_{23}$ lie on the edges $(v^+_1,v^-_1),$ $(v^+_1,v^+_2)$ and $(v^+_3,v^-_3)$ respectively and the points $q_{17},q_{19}$ lie on the edge $(v^+_2,v^-_2)$. Moreover, the points $q_{13},$ $q_{14},$ $q_{15},$ $q_{22}$ lie on the interior of the face $(v^+_1,v^-_1,v^-_3,v^+_3)$,
the point $q_{16}$ lies on the interior of the face $(v^+_2,v^-_2,v^-_3,v^+_3)$, the point $q_{18}$ lies on the interior of the face
$(v^+_1,v^-_1,v^-_2,v^+_2)$ and the points $q_{21}$ lies on the interior of the top $(v^+_1,v^+_2,v^+_3)$.
We need to add other three points in the interior of the prism $\mathbf{\Sigma}_{11}$ which are used to define the cone-polygon,
$$ \begin{array}{l}
q_{24}=(-0.16+0.74i,1.4),\\
q_{25}=(0.328-0.28i,1.99),\\
q_{26}=(0.325+0.29i,4.652).
\end{array}
$$

We verify the location of all these points as follows:

$\bullet$\ The point $q_{12}$ is in the intersection of the interiors of  $R^{(11)}_1R^{(11)}_3R^{(11)}_2(\mathcal{S}_0)$ and $\mathcal{S}_0$;

$\bullet$\ The point $q_{13}$ is in the intersection of the interiors of $\mathcal{S}_0$, $R^{(11)}_1R^{(11)}_2(\mathcal{S}_0)$ and $R^{(11)}_1R^{(11)}_3R^{(11)}_2(\mathcal{S}_0)$;

$\bullet$\ The point $q_{14}$ is in the intersection of the interiors of $\mathcal{S}_0$, $R^{(11)}_2(\mathcal{S}_0)$ and $R^{(11)}_1R^{(11)}_2(\mathcal{S}_0)$;

$\bullet$\ The point $q_{15}$ is in the intersection of the interiors of $\mathcal{S}_0$, $R^{(11)}_2(\mathcal{S}_0)$, $R^{(11)}_3(\mathcal{S}_0)$ and $R^{(11)}_3R^{(11)}_2(\mathcal{S}_0)$;

$\bullet$\ The points $q_{16},q_{19},q_{20}$ are in the intersection of the interiors of $R^{(11)}_2(\mathcal{S}_0)$ and $R^{(11)}_3(\mathcal{S}_0)$;

$\bullet$\ The point $q_{17}$ is in the intersection of the interiors of $\mathcal{S}_0$ and $R^{(11)}_2(\mathcal{S}_0)$;

$\bullet$\ The point $q_{18}$ is in the intersection of the interiors of $\mathcal{S}_0$, $R^{(11)}_2(\mathcal{S}_0)$ and $R^{(11)}_1R^{(11)}_3R^{(11)}_2(\mathcal{S}_0)$;

$\bullet$\ The point $q_{21}$ is in the intersection of the interiors of $T^{(11)}(\mathcal{S}_0)$, $R^{(11)}_2(\mathcal{S}_0)$ and $R^{(11)}_3(\mathcal{S}_0)$;

$\bullet$\ The point $q_{22}$ is in the intersection of the interiors of $R^{(11)}_1R^{(11)}_2(\mathcal{S}_0)$, $R^{(11)}_2(\mathcal{S}_0)$ and $R^{(11)}_3(\mathcal{S}_0)$;

$\bullet$\ The point $v^{+}_0$ is in the intersection of the interiors of $R^{(11)}_1R^{(11)}_2(\mathcal{S}_0)$, $T^{(11)}(\mathcal{S}_0)$ and $R^{(11)}_2(\mathcal{S}_0)$;

$\bullet$\ The point $q_{23}$ is in the intersection of the interiors of $\mathcal{S}_0$, $R^{(11)}_3(\mathcal{S}_0)$ and $R^{(11)}_3R^{(11)}_2(\mathcal{S}_0)$;

$\bullet$\ The point $q_{24}$ is in the intersection of the interiors of $\mathcal{S}_0,$ $R^{(11)}_2(\mathcal{S}_0),$ $R^{(11)}_3(\mathcal{S}_0)$ and $R^{(11)}_3R^{(11)}_2(\mathcal{S}_0)$;

$\bullet$\ The point $q_{25}$ is in the intersection of the interiors of $\mathcal{S}_0,$ $R^{(11)}_2(\mathcal{S}_0),$ $ R^{(11)}_1R^{(11)}_2(\mathcal{S}_0)$ and $R^{(11)}_1R^{(11)}_3R^{(11)}_2(\mathcal{S}_0)$;

$\bullet$\ The point $q_{26}$ is in the intersection of the interiors of $T^{(11)}(\mathcal{S}_0),$ $R^{(11)}_2(\mathcal{S}_0),$ $R^{(11)}_3(\mathcal{S}_0)$ and $R^{(11)}_1R^{(11)}_2(\mathcal{S}_0)$.

In term of these, we denote by $X_1$ a closed polygonal chain joining in order with the points $p_{12},$ $p_{13},$ $p_{14},$ $p_{15},$ $p_{16},$ $p_{17},$ $p_{18}$ and denote by $X_2$ a closed polygonal chain joining in order with the points $p_{16},$ $p_{19},$ $p_{20},$ $p_{21},$ $v^+_0$, $p_{22},$ $p_{24}$. So then we can define two cone-polygons $\mathbb{D}_{q_{25}}(X_1)$ and $\mathbb{D}_{q_{26}}(X_2)$. By examining the locations of these points, we show that $\mathbb{D}_{q_{25}}(X_1)$ is in the intersection of the interiors of $\mathcal{S}_0$, $\mathcal{U}_1$ and $\mathbb{D}_{q_{26}}(X_2)$ is in the intersection of the interiors of $R^{(11)}_2(\mathcal{S}_0)$ and
$T^{(11)}(\mathcal{S}_0),$ $R^{(11)}_1R^{(11)}_2(\mathcal{S}_0),$ $R^{(11)}_3(\mathcal{S}_0)$, namely, the intersection of the interiors of $\mathcal{U}_1$ and $\mathcal{U}_2$. The remaining faces can be easily verified which are contained inside
$\mathcal{S}_0,\mathcal{U}_1$ or $\mathcal{U}_2$ .

Finally, we define three polyhedral cells as follows:

(i)\ The polyhedral cell $\mathbb{P}_1$ is defined by the faces $\mathbb{D}_{q_{25}}(X_1),$ $\mathbb{D}_{q_{26}}(X_2),$ $(v^+_1,q_{12},q_{18},q_{17},q_{19},q_{20})$, $(v_1^+,q_{12},q_{13},q_{14},q_{15},q_{22},v^+_0)$ and $(v^+_1,q_{20},q_{21},v^+_0)$  as its boundary;

(ii)\ The polyhedral cell $\mathbb{P}_2$ is defined by the faces $\mathbb{D}_{q_{26}}(X_2),$ $(q_{15},q_{23},q_{24}),$ $(v^+_1,q_{20},q_{21},v^+_0)$ $(v^+_1,q_{12},q_{18},q_{17},q_{19},q_{20})$, $(v_1^+,q_{12},q_{13},q_{14},q_{15},q_{22},v^+_0)$ and  $(q_{23},q_{16},q_{24})$ as its boundary;

(iii)\ The polyhedral cell $\mathbb{P}_3$ is defined by the faces $\mathbb{D}_{q_{25}}(X_1),$ $(q_4,q_5,q_7),$  $(q_5,q_7,q_8),$ $(q_4,q_5,v^-_0),$ $(q_{15},q_{23},q_{24}),$ $(q_8,q_9,v^-_3,q_{23},q_{16},q_{17})$, $(q_{23},q_{16},q_{24})$, $(q_4,q_7,q_8,q_{17},q_{18},q_{12})$, $(v^-_3,v^-_0,q_5,q_9)$ and $(q_{12},q_{13},q_{14},q_{15},q_{23},v^-_3,v^-_0,q_4)$ as its boundary;

By Lemma 4.1 and the properties of star-convex of $\mathcal{U}_1$ and $\mathcal{U}_2$, we conclude that the polyhedral cell $\mathbb{P}_1$ contained inside $\mathcal{U}_1$; the polyhedral cell $\mathbb{P}_2$ contained inside $\mathcal{U}_2$; the polyhedral cell $\mathbb{P}_3$ is contained inside $\mathcal{S}_0$. This completes the proof.
\end{proof}


\medskip


\begin{thebibliography}{99}


\bibitem{Co}
    P. M. Cohn,
    A presentation for $SL_2$ for Euclidean quadratic imaginary number fields,  \emph{Mathematika}. \textbf{15}(1968), 156¨C163.

\bibitem{DFP}
    M. Deraux, E. Falbel and J. Paupert.
    New constructions of fundamental polyhedra in complex hyperbolic
space. \emph{Acta Math.} \textbf{194}(2005), no. 2, 155-201.

\bibitem{Fi}
    B. Fine,  Algebraic theory of the Bianchi groups, \emph{Marcel Dekker Inc.} (1989).


\bibitem{FFLP}
    E. Falbel, G. Francsics, P. Lax and J. R. Parker.
    Generators of a Picard modular group in two complex dimensions.
    \emph{to appear in proc. AMS.}

  \bibitem{FP}
    E. Falbel and J. R. Parker.
    The Geometry of Eisenstein-Picard Modular Group.
    \emph{Duke Math. J.}
    \textbf{131}(2006), no. 2, 249-289.

  \bibitem{FFP}
    E. Falbel, G. Francsics and J. R. Parker.
    The geometry of the Gauss-picard modular group.
    \emph{Math Annalen}. (pubilished online: 04 May 2010).


  \bibitem{FL1}
     G. Francsics and P. Lax.
    A semi-explicit fundamental domain for the Picard Modular Group in complex hyperbolic
space.
    \emph{Contemporary Mathematics.}
    \textbf{368}(2005), 211-226.

  \bibitem{FL2}
    G. Francsics and P. Lax.
    An explicit fundamental domain for the Picard Modular Group in two complex
dimensions.   (Preprint 2005).

  \bibitem{Go}
  W. M. Goldman.
    \emph{Complex Hyperbolic Geometry.}
    (Oxford Mathematical Monographs, Oxford University Press 1999).

 \bibitem{GP}
    W. M. Goldman and J. R. Parker, Complex hyperbolic ideal triangle groups, \emph{J. reine angewandte
Math.} \textbf{425}(1992), 71-86.

\bibitem{GR}
    H. Garland and M. S. Raghunathan.
    Fundamental domains for lattices in ($\mathbf{R}$-)rank 1 semisimple Lie groups. \emph{Ann. of Math.}
     \textbf{92}(1970), no.2, 279-326..

  \bibitem{Mo}
    G. D. Mostow.
   On a remarkable class of polyhedra in complex hyperbolic
space.
    \emph{ Pacific J. Mathematics}. \textbf{86}(1980), 171-276.

    \bibitem{Pa}
 J .R. Parker.
  Complex hyperbolic lattices.
    \emph{Contemporary Mathematics.}
    \textbf{501}(2009), 1-42.


     \bibitem{Sc}
  G. P. Scott.
  The geometries of 3-manifolds.
    \emph{Bull. London Math. Soc.} \textbf{15}(1983), 401-487.

 \bibitem{Sch}
  R. E. Schwartz.
  Complex hyperbolic triangle groups. \emph{Proc. of the intern. Congress of Mathematicians}: Invited Lectures
\textbf{1}(2002), 339-350.

    \bibitem{ST}
  I. N. Stewart and D. O. Tall,
  Algebraic number theory.\emph{Chapman and Hall Ltd.} (1979).

\bibitem{Sw}
	 R.G. Swan,
    Generators and relations for certain special linear groups, \emph{Adv. in Math.} \textbf{6}(1971), 1-77.

 \bibitem{Zi}
  T. Zink.
  \"{U}er die Anzahl der Spitzen einiger arithmetischer Untergruppen unit\"{a}er Gruppen. \emph{Math.
Nachr.} \textbf{89} (1979), 315-320.

 \bibitem{Zh}
  T. Zhao.
  A minimal volume arithmetic cusped complex hyperbolic orbifold. (Preprint).




\end{thebibliography}
\end{document}